\documentclass[10pt]{amsart}

\usepackage{latexsym,exscale,enumerate,amsfonts,amssymb, ulem, xparse, mathtools}
\usepackage{amsmath,amsthm,amsfonts,amssymb,amscd, stmaryrd,textcomp}
\usepackage{hyperref}
\usepackage{ulem}
\usepackage{bbold}
\usepackage[usenames,dvipsnames]{xcolor}
\colorlet{green}{black!30!green} 

\usepackage{graphicx}

\usepackage{tikz}
\usetikzlibrary{decorations.markings}
\usetikzlibrary{decorations.pathreplacing}
\usetikzlibrary{arrows,shapes,positioning}
\tikzstyle directed=[postaction={decorate,decoration={markings,
    mark=at position #1 with {\arrow{>}}}}]
\tikzstyle rdirected=[postaction={decorate,decoration={markings,
    mark=at position #1 with {\arrow{<}}}}]
\tikzset{anchorbase/.style={baseline={([yshift=-0.5ex]current bounding box.center)}},
  arrowinthemiddle/.style={postaction=decorate,decoration={markings,mark=at position 0.5 with {\arrow{>}}}},
arrowinthemiddlerev/.style={postaction=decorate,decoration={markings,mark=at position 0.5 with {\arrow{<}}}},
cross line/.style={preaction={draw=white,line width=4pt,-}},
int/.style={very thick},
zero/.style={thin,dotted},
uno/.style={thin},
aboverotated/.style={above,rotate=60,anchor=west},
belowrotated/.style={below,rotate=60,anchor=east}}

\usetikzlibrary{knots}

\newcommand{\midarrow}{node[midway,sloped,allow upside down] {\hspace{0.05cm}\tikz[baseline=0] \draw[->] (0,0) -- +(.001,0);}}
\newcommand{\midarrowrev}{node[midway,sloped,allow upside down] {\hspace{0.05cm}\tikz[baseline=0] \draw[-<] (0,0) -- +(.001,0);}}
\newcommand{\nendarrow}{node[near end,sloped,allow upside down] {\hspace{0.05cm}\tikz[baseline=0] \draw[->] (0,0) -- +(.001,0);}}

\newcommand{\nendarrowrev}{node[near end,sloped,allow upside down] {\hspace{0.05cm}\tikz[baseline=0] \draw[-<] (0,0) -- +(.001,0);}}

\newcommand{\startarrowrev}{node[at start,sloped,allow upside down] {\hspace{0.05cm}\tikz[baseline=0] \draw[-<] (0,0) -- +(.001,0);}}

\tikzset{smallnodes/.style={every node/.style={font=\footnotesize}}}

\newcommand{\C}{\mathbb{C}}
\newcommand{\R}{\mathbb{R}}
\newcommand{\Z}{\mathbb{Z}}
\newcommand{\N}{\mathbb{N}}
\newcommand{\fieldk}{\mathbb{k}}

\newcommand{\Ss}{\mathbb{S}}
\newcommand{\gl}{\mathfrak{gl}}

\DeclareMathOperator{\Hom}{Hom}

\newcommand{\abs}[1]{\left|#1\right|}

\newcommand{\mapto}{\rightarrow}

\newcommand{\counit}{\boldsymbol{\mathrm{u}}}
\newcommand{\uBr}{\mathsf{Br}}
\newcommand{\sln}{\mathfrak{sl}_n}
\newcommand{\slnn}[1]{\mathfrak{sl}_{#1}}
\newcommand{\catRep}{{\mathsf{Rep}}}
\newcommand{\calH}{{\mathcal{H}}}
\newcommand{\calS}{{\mathcal{S}}}
\newcommand{\1}{\boldsymbol{1}}

\DeclareMathOperator{\End}{End}
\DeclareMathOperator{\Kar}{Kar}

\newtheorem{theorem}{Theorem}
\newtheorem{definition}[theorem]{Definition}
\newtheorem{example}[theorem]{Example}
\newtheorem{proposition}[theorem]{Proposition}
\newtheorem{corollary}[theorem]{Corollary}
\newtheorem{remark}[theorem]{Remark}
\newtheorem{lemma}[theorem]{Lemma}

\newcommand{\catTangles}{{\mathsf{Tangles}}}

\begin{document}
%

\title
    {
      Polynomial link invariants and quantum algebras
}

\date{}

\author{Hoel Queffelec}
\address{IMAG, U. Montpellier, CNRS, Montpellier, France}
\email{hoel.queffelec@umontpellier.fr}

\begin{abstract}
The definition of the Jones polynomial in the 80's gave rise to a large family of so-called quantum link invariants, based on quantum groups. These quantum invariants are all controlled by the same two-variable invariant (the HOMFLY-PT polynomial), which also specializes to the older Alexander polynomial.
Building upon quantum Schur--Weyl duality and variants of this phenomenon, I will explain an algebraic setup that allows for global definitions of these quantum polynomials, and discuss extensions of these quantum objects designed to encompass all of the mentioned invariants, including the HOMFLY-PT polynomial.
\end{abstract}

\maketitle

\section*{Introduction}

A large part of link theory consisted and still consists in the development and study of link invariants: these are functions, typically defined on knot diagrams, that only depend on the isotopy class of the knot on which they are evaluated and not on the particular choice of representative used to compute the result. This ensures that two knot diagrams for which a given invariant takes different values represent genuinely different knots.

Polynomial knot invariants have played a central historical role, and can often be linked to two main families that arose either from the Alexander polynomial \cite{Alexander} or from the Jones polynomial \cite{Jones}. Both of these invariants can be obtained by specialization of a two-parameter more general invariant called the HOMFLY-PT polynomial \cite{HOMFLY,PT}.

The purpose of these notes is to explain and illustrate the relationship between these invariants and several objects that appear in quantum algebra. For the Jones polynomial, this relationship is basically the starting point of quantum topology and was present almost from the very beginning, through the work of Reshetikhin, Turaev and Witten \cite{RT,Witten3}. Indeed, Reshetikhin and Turaev reformulated Jones construction by assigning to a tangle an intertwiner in some category of representations of a special linear quantum group. In the case of a knot, one obtains an endomorphism of the trivial object in the category. Such an endomorphism is a scalar multiple of the identity, and this scalar is the Jones polynomial, or generalizations of it, depending on the choice of the quantum group.

Although the Alexander polynomial was at first defined in a rather different context, quantum reformulations of it have been known for some time now~\cite{ViroAlexander,KauffmanSaleur,MurakamiAlexander} (see also \cite{SarAlexander} for a concise treatment of it). The HOMFLY-PT polynomial, on the other hand, does not enjoy a direct translation in quantum terms, or at least not by the use of intertwiners for a certain quantum group.

After presenting the different invariants that will play a role here, first diagrammatically and then in relation with the theory of representations of $U_q(\gl_{m|n})$, I will recover from these constructions two quantum algebras (the Hecke algebra and the quantum Schur algebra) and try to convey the idea that they can serve as algebraic rigidifying tools, or in other words better replacements, for the categories of representations. Furthermore, these objects, contrarily to the categories of representations, can be extended to encompass the HOMFLY-PT polynomial as well, which I'll briefly illustrate at the end of these notes.

\noindent {\bf Organization:} Section~\ref{sec:knots} is devoted to brief definitions of the notions of knots, links and tangles, and to their invariants. These invariants are defined using the so-called quantized oriented Brauer category. Section~\ref{sec:RT} relates this to the original Reshetikhin-Turaev approach by defining a functor between the quantized oriented Brauer category and the category of representations of the quantum group $U_q(\gl_{m|n})$. Finally, Section~\ref{sec:SW} is devoted to Schur--Weyl and skew Howe dualities and the role they recently played in knot theory.

The presentation given in these notes is far from being historically accurate: I have rather tried to define all invariants in a unified way and use this definition to recover several quantum constructions. In many situations, the story actually went the other way around. However, I hope that the presentation I chose to follow will make the relations and the central role that quantum algebras play in this picture more apparent, and that these notes can serve as a motivation and a general illustration before reading more involved references.

\noindent {\bf Acknowledgements:} These lecture notes follow a lecture series given at Winterbraids IX in Reims in March 2019. I would like to warmly thank the organizers (Paolo Bellingeri, Vincent Florens, Jean-Baptiste Meilhan, Lo\"ic Poulain d'Andecy, Emmanuel Wagner) for their invitation, their support, and more generally for putting together these great winter schools year after year. Many thanks also to Antonio Sartori for his comments on a preliminary version of these notes and for teaching me most of what's now in these notes, and to the anonymous referee for her/his helpful comments.

\section{Knots, links, tangles and their invariants} \label{sec:knots}

\subsection{Knots and links}

Here I briefly and informally recall the notions of knot, link and tangle. Precise definitions and detailed discussions can be found for example in~\cite{BurdeZieschang,Cromwell,Rolfsen}.

A knot is an embedding of the circle: $\Ss^1\hookrightarrow \R^3$, up to ambient isotopy. A link is an embedding of a finite number of copies of the circle: $\cup_{k} \Ss^1\hookrightarrow \R^3$, also considered up to global isotopy (thus the empty link is a link and a knot is a link). Choosing a direction of projection, one can generically represent a knot by a diagram, as in the following example.

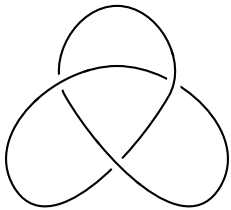
\begin{example}
\[
\begin{tikzpicture}[anchorbase,scale=.8]
  \begin{knot}[
            clip width=7,
            consider self intersections,
            end tolerance=1pt,
            flip crossing=2,
    ]
    \strand [thick] (0,2) to [out=0,in=60] (.86,.5) to [out=-120,in=-60] (-1.72,-1) to [out=120,in=180] (0,1) to [out=0,in=60] (1.72,-1) to  [out=-120,in=-60] (-.86,.5) to [out=120, in=180] (0,2);
    \end{knot}
\end{tikzpicture}
\]
\end{example}

Most generally, the knots and links we will consider in what follows will be framed: they can be thought of as being made of ``bands'' rather than strings. This does not cause much difference in the theory, and in our case this is purely technical: the invariants that will show up later are naturally invariants of framed knots, but making them invariants of genuine knots is just a matter of a rescaling factor based on a crossing count.

A given knot can be represented by many different diagrams, and the following classical theorem of Reidemeister gives us a combinatorial tool to relate different diagrams associated to the same knot.

\begin{theorem}
  Two diagrams $D$ and $D'$ represent the same framed knot if and only if they are related by a sequence of planar isotopies and of moves of the following kinds:
  \begin{itemize}
  \item $R_I'$:
    $
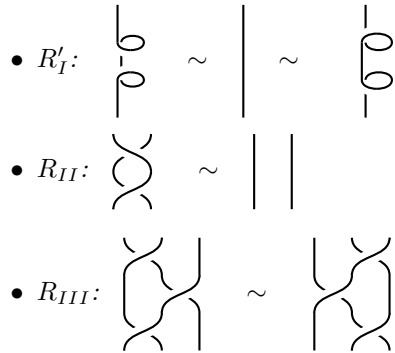
    \begin{tikzpicture}[anchorbase,scale=.5]
\begin{knot}[
            clip width=7,
            consider self intersections,
            end tolerance=1pt,
            flip crossing=2,
    ]
    \strand [thick] (0,0) -- (0,1) to [out=90,in=90] (.7,1) to [out=-90,in=-90] (.1,1) -- (.1,2) to [out=90,in=90] (.7,2) to [out=-90,in=-90] (0,2) -- (0,3);
    \end{knot}
    \end{tikzpicture} \quad \sim \quad
    \begin{tikzpicture}[anchorbase,scale=.5]
    \draw [thick] (0,0) -- (0,3);
    \end{tikzpicture}
    \quad \sim \quad
        \begin{tikzpicture}[anchorbase,scale=.5]
\begin{knot}[
            clip width=7,
            consider self intersections,
            end tolerance=1pt,
            flip crossing=1,
    ]
    \strand [thick] (0,0) -- (0,1) to [out=90,in=90] (.7,1) to [out=-90,in=-90] (-.1,1) -- (-.1,2) to [out=90,in=90] (.7,2) to [out=-90,in=-90] (0,2) -- (0,3);
    \end{knot}
    \end{tikzpicture}
    $
      \item $R_{II}$: $
            \begin{tikzpicture}[anchorbase,scale=.5]
\begin{knot}[
            clip width=7,
            consider self intersections,
            end tolerance=1pt,
            flip crossing=,
    ]
  \strand [thick] (0,0) to [out=90,in=-90] (1,1) to [out=90,in=-90] (0,2);
  \strand [thick] (1,0) to [out=90,in=-90] (0,1) to [out=90,in=-90] (1,2);
    \end{knot}
            \end{tikzpicture}
            \quad \sim \quad
            \begin{tikzpicture}[anchorbase,scale=.5]
              \draw [thick] (0,0) --(0,2);
              \draw [thick] (1,0) -- (1,2);
            \end{tikzpicture}
            $
          \item $R_{III}$:
            $
            \begin{tikzpicture}[anchorbase,scale=.5]            
              \begin{knot}[
                  clip width=7,
                  consider self intersections,
                  end tolerance=1pt,
                  flip crossing=,
                ]
                \strand [thick] (0,0) to [out=90,in=-90] (1,1) to [out=90,in=-90] (2,2) -- (2,3);
                \strand [thick] (1,0) to [out=90,in=-90] (0,1) -- (0,2) to [out=90,in=-90] (1,3);
                \strand [thick] (2,0) -- (2,1) to [out=90,in=-90] (1,2) to [out=90, in=-90] (0,3);
              \end{knot}
            \end{tikzpicture}
            \quad \sim \quad
            \begin{tikzpicture}[anchorbase,scale=.5]            
              \begin{knot}[
                  clip width=7,
                  consider self intersections,
                  end tolerance=1pt,
                  flip crossing=,
                ]
                \strand [thick] (0,0) -- (0,1) to [out=90,in=-90] (1,2) to [out=90,in=-90]  (2,3);
                \strand [thick] (1,0) to [out=90,in=-90] (2,1) --(2,2) to [out=90,in=-90] (1,3);
                \strand [thick] (2,0) to [out=90,in=-90] (1,1) to [out=90, in=-90] (0,2) --(0,3);
              \end{knot}
            \end{tikzpicture}
            $
    \end{itemize}
  \end{theorem}

In what follows, we will very often consider oriented versions of knots and links. This can be represented by orienting the diagram, and causes no difference in Reidemeister's theorem.

\subsection{Tangles and their category presentation}

Tangles are links that are allowed to have ends. More precisely, they are embeddings
\[
\amalg_k \Ss^1\amalg_l [0,1] \hookrightarrow [0,1]^3
\]
with prescribed endpoints on $[0,1]^2\times \{0,1\}$, considered up to isotopy relative boundary. Projection from the cube to the square that collapses, say, the first coordinate, allows to consider diagrams just as in the case of links.

\begin{example}
  \[
\begin{tikzpicture}[anchorbase,scale=.8]
  \begin{knot}[
            clip width=7,
            consider self intersections,
            end tolerance=1pt,
            flip crossing=1,
            flip crossing=2,
            flip crossing=4,
            flip crossing=6,
            flip crossing=8,
    ]
    \strand [thick] (0,2) to [out=0,in=60] (.86,.5) to [out=-120,in=-60] (-1.72,-1) to [out=120,in=180] (0,1) to [out=0,in=60] (1.72,-1) to  [out=-120,in=-60] (-.86,.5) to [out=120, in=180] (0,2);
    \strand [thick] (-2,-2.5) to [out=90,in=180] (-1.72,-.8) to [out=0,in=180] (0,-1.5) to [out=0,in=180] (1.72,-.8) to [out=0,in=90] (2,-2.5);
    \strand [thick] (-1,3) to [out=-90,in=180] (0,1.5) to [out=0,in=-90] (1,3);
  \end{knot}
\end{tikzpicture}
  \]
\end{example}

We will often refer to a tangle by saying that it is a $(k,l)$ tangle. By this we mean that it has $k$ ends at the bottom ($[0,1]^2\times \{0\}$) and $l$ ends at the top ($[0,1]^2\times \{1\}$).

Braids appear as a particular class of tangles, namely those with no horizontal tangencies. This is probably neither the easiest nor the most convenient definition, and it might be more illustrative to define the braid group on $n$ strands as the group generated by the elements:
\[
\sigma_i=
 \begin{tikzpicture}[anchorbase,scale=.5]
  \begin{knot}[
      clip width=7,
      consider self intersections,
      end tolerance=1pt,
      flip crossing=,
    ]
    \draw [thick] (-1.5,0) -- (-1.5,1);
    \node at (-.75,.5) {$\cdots$};
    \strand [thick] (0,0) to [out=90,in=-90] (1,1);
    \strand [thick] (1,0) to [out=90,in=-90] (0,1);
    \node at (1.75,.5) {$\cdots$};
    \draw [thick] (2.5,0) -- (2.5,1);
    \node at (-1.5,-.3) {\tiny $1$};
    \node at (0,-.3) {\tiny $i$};
    \node at (1,-.3) {\tiny $i$+$1$};
    \node at (2.5,-.3) {\tiny $n$};
  \end{knot}
\end{tikzpicture}
 \]
The multiplication law is given by vertical superposition, and the inverses of the generators can be guessed from the Reidemeister II move. To get a presentation by generators and relations, one should also add the third Reidemeister relation for adjacent crossings, and isotopies that allow far-away crossings to commute.

If braids appear as a particular class of tangles, which themselves generalize links, one can also go back to link theory from braid theory. Indeed, the annular closure (see the right-hand side of Equation \eqref{eq:betahat}) of a braid always forms a link. The following theorem of Alexander states that any link actually appears in this way, and we will make use of it later.

\begin{theorem}
  Every knot is the closure of a braid:
\begin{equation} \label{eq:betahat}
  K\quad \sim \quad
  \begin{tikzpicture}[anchorbase,scale=.7]
    \draw (0,0) rectangle (2,2);
    \node at (1,1) {$\beta$};
    \draw (.5,2) to [out=90, in=90] (3.5,2) -- (3.5,0) to [out=-90,in=-90] (.5,0);
    \draw (1.5,2) to [out=90, in=90] (2.5,2) -- (2.5,0) to [out=-90,in=-90] (1.5,0);
    \node at (3,1) {\small $\cdots$};
    \end{tikzpicture}
\end{equation}
\end{theorem}

It is worth mentioning that this correspondence is not one-to-one: different braids can have isotopic closures, but moves relating such braids have been classified (this is called Markov's theorem).

A nice feature of the tangles is that they can easily be organized into a category, which we will much use in the next sections.

\begin{definition}(See also \cite[Section XII]{Kassel}.)
  The category \(\catTangles\) is the monoidal category of oriented, framed tangles. Its objects are generated by \(\{\uparrow,\downarrow\}\), and its morphisms are (diagrams of) oriented, framed tangles modulo isotopy. We also define \(\catTangles_\fieldk\) as its \(\fieldk\)-linear version, with the same objects but with morphisms being  free \(\fieldk\)-modules (with $\fieldk$ any ring) generated by tangles.
\end{definition}

In other words, objects in both categories are sequences of up and down arrows, while morphisms are given by tangles mapping between such sequences, or linear combinations of such tangles. The adjective {\it monoidal} means that the category is equipped with a tensor product, which in our case simply translates into the fact that two tangles can be put next to each other, forming a bigger tangle.

\begin{example}
  Choosing $\eta =\emptyset$ and $\eta'=\uparrow \downarrow \downarrow \uparrow$, then:
  \[
  \begin{tikzpicture}[anchorbase,scale=.5]
    \draw [->] (1,1) -- (1,.5) to [out=-90,in=-90] (0,.5) -- (0,1);
    \draw [->] (2,1) -- (2,.5) to [out=-90,in=-90] (3,.5) -- (3,1);
    \end{tikzpicture}
\quad   and \quad 
  \begin{tikzpicture}[anchorbase,scale=.5]
  \begin{knot}[
      clip width=7,
      consider self intersections,
      end tolerance=1pt,
      flip crossing=,
    ]
    \strand [->] (1,1) -- (1,.5) to [out=-90,in=-90] (3,.5) -- (3,1);
    \strand [->] (2,1) -- (2,.5) to [out=-90,in=-90] (0,.5) -- (0,1);
  \end{knot}
\end{tikzpicture}
\]
  are typical morphisms in $\catTangles(\eta,\eta')$, and
\[
  2\;
   \begin{tikzpicture}[anchorbase,scale=.5]
    \draw [->] (1,1) -- (1,.5) to [out=-90,in=-90] (0,.5) -- (0,1);
    \draw [->] (2,1) -- (2,.5) to [out=-90,in=-90] (3,.5) -- (3,1);
    \end{tikzpicture}  
\; - 3 \;
  \begin{tikzpicture}[anchorbase,scale=.5]
    \begin{knot}[
      clip width=7,
      consider self intersections,
      end tolerance=1pt,
      flip crossing=,
    ]
    \strand [->] (1,1) -- (1,.5) to [out=-90,in=-90] (3,.5) -- (3,1);
    \strand [->] (2,1) -- (2,.5) to [out=-90,in=-90] (0,.5) -- (0,1);
  \end{knot}
\end{tikzpicture}
\]
  something that typically leaves in $\catTangles_\fieldk(\eta,\eta')$.  Notice that:
 \[
  \begin{tikzpicture}[anchorbase,scale=.5]
    \draw [->] (1,1) -- (1,.5) to [out=-90,in=-90] (0,.5) -- (0,1);
    \draw [->] (2,1) -- (2,.5) to [out=-90,in=-90] (3,.5) -- (3,1);
  \end{tikzpicture}
  \;=\;
  \begin{tikzpicture}[anchorbase,scale=.5]
    \draw [->] (1,1) -- (1,.5) to [out=-90,in=-90] (0,.5) -- (0,1);
  \end{tikzpicture}
  \;\otimes \;
  \begin{tikzpicture}[anchorbase,scale=.5]
    \draw [->] (2,1) -- (2,.5) to [out=-90,in=-90] (3,.5) -- (3,1);
  \end{tikzpicture}
  \]
  while
  \[
  \begin{tikzpicture}[anchorbase,scale=.5]
  \begin{knot}[
      clip width=7,
      consider self intersections,
      end tolerance=1pt,
      flip crossing=,
    ]
    \strand [->] (1,1) -- (1,.5) to [out=-90,in=-90] (3,.5) -- (3,1);
    \strand [->] (2,1) -- (2,.5) to [out=-90,in=-90] (0,.5) -- (0,1);
  \end{knot}
\end{tikzpicture}
\]
 is not the tensor product of two simpler tangles.
\end{example}

In light of the above example, here might be a good place to highlight the reading conventions we adopt for diagrammatic categories: objects are displayed horizontally and read (with respect to tensor product) from left to right, and morphisms appear vertically and are read from bottom to top.

\subsection{The invariants}

Here and throughout these notes, we set $\fieldk$ to be either $\C(q)[q^{\pm \beta}]$, or $\C(q)$. In the first case, this is to be understood as an extension of $\C(q)$ by a formal invertible variable denoted $q^{\beta}$. Then setting $\beta=d\in \Z$ induces a map $\C(q)[q^{\pm \beta}]\rightarrow \C(q)$.

Our interest in this peculiar writing is that we will allow ourselves to write things like: $q^2\cdot q^{2\beta}=q^{2+2\beta}$, which makes the following notations easier to handle.
\begin{align}
  [x]&=\frac{q^x-q^{-x}}{q-q^{-1}}, \quad x\in \Z\beta+\Z \\
  [n]!&=[1][2]\cdots [n],\quad n\in \N
\end{align}

\begin{example}
  \[
    [2]=q+q^{-1},\quad [\beta]=\frac{q^\beta-q^{-\beta}}{q-q^{-1}}
    \]
  \end{example}

The following definition despite its simplicity, will be instrumental in our presentation of the invariants. This definition originates from \cite{DDS} up to minor changes of conventions.

\begin{definition}
  The quantized oriented Brauer category $\uBr(\beta)$ is the quotient of $\catTangles_\fieldk$, with $\fieldk=\C(q)[q^{\pm \beta}]$,    modulo the following relations
\begin{subequations}
  \label{eq:5}
  \begin{align}
     \begin{tikzpicture}[smallnodes,anchorbase,xscale=0.7,yscale=0.5]
   \draw[uno] (1,0)  -- ++(0,0.3) \midarrow .. controls ++(0,0.7) and ++(0,-0.7) .. ++(-1,1.4) -- ++(0,0.3)  \midarrow ;
   \draw[uno,cross line] (0,0)  -- ++(0,0.3) \midarrow .. controls ++(0,0.7) and ++(0,-0.7) .. ++(1,1.4) -- ++(0,0.3)  \midarrow ;
    \end{tikzpicture} \; - \;
    \begin{tikzpicture}[smallnodes,anchorbase,xscale=0.7,yscale=0.5]
   \draw[uno] (0,0)  -- ++(0,0.3) \midarrow .. controls ++(0,0.7) and ++(0,-0.7) .. ++(1,1.4) -- ++(0,0.3)  \midarrow ;
   \draw[uno,cross line] (1,0)  -- ++(0,0.3) \midarrow .. controls ++(0,0.7) and ++(0,-0.7) .. ++(-1,1.4) -- ++(0,0.3)  \midarrow ;
    \end{tikzpicture} \; & = (q^{-1}-q) \;
    \begin{tikzpicture}[smallnodes,anchorbase,yscale=0.5,xscale=0.7]
   \draw[uno] (1,0) .. controls ++(0.25,0.5) and ++(0.25,-0.5)  .. ++(0,2) \midarrow ;
   \draw[uno] (2,0) .. controls ++(-0.25,0.5) and ++(-0.25,-0.5)  .. ++(0,2) \midarrow  ;
    \end{tikzpicture}\;,  \qquad&
        \begin{tikzpicture}[smallnodes,anchorbase,scale=0.7]
      \draw[uno] (0,0) arc (0:360:0.5cm) \midarrowrev;
    \end{tikzpicture} \; =\;
        \begin{tikzpicture}[smallnodes,anchorbase,scale=0.7]
      \draw[uno] (0,0) arc (0:360:0.5cm) \midarrow;
    \end{tikzpicture} \; &= [\beta] \label{eq:52}\\
    \begin{tikzpicture}[smallnodes,anchorbase,xscale=0.7,yscale=0.5]
   \draw[uno] (0.8,1)  .. controls ++(0,-0.3) and ++(0.3,0) .. ++(-0.3,-0.6)  .. controls ++(-0.5,0) and ++(0,-0.7) .. ++(-0.5,1.6) \nendarrow  ;
   \draw[uno, cross line] (0.8,1)  .. controls ++(0,0.3) and ++(0.3,0) .. ++(-0.3,0.6) \startarrowrev .. controls ++(-0.5,0) and ++(0,0.7) .. ++(-0.5,-1.6) \nendarrowrev;
    \end{tikzpicture} \; = \;
    \begin{tikzpicture}[smallnodes,anchorbase,xscale=0.7,yscale=0.5]
   \draw[uno, cross line] (-0.8,1)  .. controls ++(0,0.3) and ++(-0.3,0) .. ++(0.3,0.6) \startarrowrev .. controls ++(0.5,0) and ++(0,0.7) .. ++(0.5,-1.6) \nendarrowrev ;
   \draw[uno, cross line] (-0.8,1)  .. controls ++(0,-0.3) and ++(-0.3,0) .. ++(0.3,-0.6)  .. controls ++(0.5,0) and ++(0,-0.7) .. ++(0.5,1.6) \nendarrow  ;
    \end{tikzpicture} \;& = q^{-\beta} \;
    \begin{tikzpicture}[smallnodes,anchorbase,xscale=0.7,yscale=0.5]
      \draw[uno] (0,0)  -- ++(0,2) \midarrow ;
    \end{tikzpicture}\;, &
    \begin{tikzpicture}[smallnodes,anchorbase,xscale=0.7,yscale=0.5]
   \draw[uno, cross line] (0.8,1)  .. controls ++(0,0.3) and ++(0.3,0) .. ++(-0.3,0.6) \startarrowrev .. controls ++(-0.5,0) and ++(0,0.7) .. ++(-0.5,-1.6) \nendarrowrev ;
   \draw[uno, cross line] (0.8,1)  .. controls ++(0,-0.3) and ++(0.3,0) .. ++(-0.3,-0.6)  .. controls ++(-0.5,0) and ++(0,-0.7) .. ++(-0.5,1.6) \nendarrow  ;
    \end{tikzpicture} \; = \;
    \begin{tikzpicture}[smallnodes,anchorbase,xscale=0.7,yscale=0.5]
   \draw[uno] (-0.8,1)  .. controls ++(0,-0.3) and ++(-0.3,0) .. ++(0.3,-0.6)  .. controls ++(0.5,0) and ++(0,-0.7) .. ++(0.5,1.6) \nendarrow  ;
   \draw[uno, cross line] (-0.8,1)  .. controls ++(0,0.3) and ++(-0.3,0) .. ++(0.3,0.6) \startarrowrev .. controls ++(0.5,0) and ++(0,0.7) .. ++(0.5,-1.6) \nendarrowrev;
    \end{tikzpicture} \; & = q^{+ \beta} \;
    \begin{tikzpicture}[smallnodes,anchorbase,xscale=0.7,yscale=0.5]
      \draw[uno] (0,0)  -- ++(0,2) \midarrow ;
    \end{tikzpicture}\label{eq:54}
  \end{align}
\end{subequations}
  \end{definition}

The category $\uBr(\beta)$ is again a monoidal category. The same definition (and consequences below) holds when $\beta$ gets specialized to $n\in \Z$. We will denote the corresponding category by $\uBr(\beta=n)$.

Dipper, Doty and Stoll proved the following result~\cite{DDS}.

\begin{proposition} \label{prop:DDS}
  $\End_{\uBr(\beta)}\left((\uparrow)^{\otimes r}(\downarrow)^{\otimes s}\right)$ is free of rank $(r+s)!$.
\end{proposition}

\begin{corollary} \label{cor:endBr}
In particular, $\End_{\uBr(\beta)}(\emptyset)\simeq \fieldk$.
\end{corollary}

\begin{proof}[Sketch of the easy part of the corollary, without referring to Proposition~\ref{prop:DDS}.]
  
  A tangle in $\End_{\uBr(\beta)}(\emptyset)$ is just a link. Using the lhs of Equation~\eqref{eq:52}, one can switch crossings until reaching a framed unlink. Equation~\eqref{eq:54} allows to reduce a framed unlink to an unframed unlink, and then the rhs of Equation~\eqref{eq:52} allows to go down to the empty link.

This proves that $\End_{\uBr(\beta)}$ is of dimension at most $1$. The harder part of the proof in~\cite{DDS} uses the existence of the HOMFLY-PT polynomial, that we are about to derive from the definition of the quantized oriented Brauer category. This is probably one of the main historical twists of these notes, which allows to give a unified presentation of the invariants at play, but causes great logical distortion!
\end{proof}

A consequence of Corollary~\ref{cor:endBr} is that given a link, its class $[L]\in \End_{\uBr(\beta)}(\emptyset)\simeq \fieldk$ yields a $\fieldk$-valued link invariant. The fact that it is indeed a link invariant simply follows from the fact that the oriented Brauer category is just a quotient of the tangle category, in which links sit.

\begin{definition}\label{def:1}
  Let \(L\) be an oriented, framed link.
  \begin{itemize}
  \item The HOMFLY--PT polynomial of \(L\) is the image $[L] \in \End_{\uBr(\beta)}(\emptyset) = \C(q)[q^{\pm \beta}]$.
  \item 
Let \(n \in \Z\). The $\sln$ Reshetikhin-Turaev polynomial of \(L\) is the image $[L] \in \End_{\uBr(\beta=n)}(\emptyset) = \C(q)$.
  \end{itemize}
\end{definition}

\begin{remark}
  The $\sln$ Reshetikhin-Turaev polynomials are really Laurent polynomials in the variable $q$. The HOMFLY-PT polynomial is not really a polynomial: one needs to localize in $q-q^{-1}$.
 \end{remark}

\begin{example}
  In the case where $\beta=2$, one recovers the Jones polynomial~\cite{Jones}. For the trefoil for example, one can compute:
  \[
\left[
    \begin{tikzpicture}[anchorbase,scale=.5]
  \begin{knot}[
            clip width=7,
            consider self intersections,
            end tolerance=1pt,
            flip crossing=2,
    ]
    \strand [thick] (0,2) to [out=0,in=60] (.86,.5) to [out=-120,in=-60] (-1.72,-1) to [out=120,in=180] (0,1) to [out=0,in=60] (1.72,-1) to  [out=-120,in=-60] (-.86,.5) to [out=120, in=180] (0,2);
    \draw [thick,->] (-.05,2) -- (.05,2);
    \end{knot}
    \end{tikzpicture}
    \right]_{\beta=2}
   = -q^3+q^{-1}+q^{-3}+q^{-5}
\]

Indeed, one first uses a skein relation to switch a crossing:
\[
\left[
    \begin{tikzpicture}[anchorbase,scale=.5]
  \begin{knot}[
            clip width=7,
            consider self intersections,
            end tolerance=1pt,
            flip crossing=2,
    ]
     \strand [thick] (0,2) to [out=0,in=60] (.86,.5) to [out=-120,in=-60] (-1.72,-1) to [out=120,in=180] (0,1) to [out=0,in=60] (1.72,-1) to  [out=-120,in=-60] (-.86,.5) to [out=120, in=180] (0,2);
     \draw [thick,->] (-.05,2) -- (.05,2);
   \end{knot}
    \end{tikzpicture}
    \right]_{\beta=2}
= 
\left[
    \begin{tikzpicture}[anchorbase,scale=.5]
  \begin{knot}[
            clip width=7,
            consider self intersections,
            end tolerance=1pt,
            flip crossing=3,
            flip crossing=2,
    ]
     \strand [thick] (0,2) to [out=0,in=60] (.86,.5) to [out=-120,in=-60] (-1.72,-1) to [out=120,in=180] (0,1) to [out=0,in=60] (1.72,-1) to  [out=-120,in=-60] (-.86,.5) to [out=120, in=180] (0,2);
     \draw [thick,->] (-.05,2) -- (.05,2);
   \end{knot}
    \end{tikzpicture}
    \right]_{\beta=2}
+(q^{-1}-q)
\left[
    \begin{tikzpicture}[anchorbase,scale=.5]
  \begin{knot}[
            clip width=7,
            consider self intersections,
            end tolerance=1pt,
            flip crossing=1,
    ]
     \strand [thick] (0,2) to [out=180,in=60] (-.86,.6) to [out=-120, in=120] (-1.72,-1) to [out=-60, in=-120] (.86,.5) to [out=60,in=0] (0,2);
     \draw [thick,->] (-.05,2) -- (.05,2);
     \strand [thick] (1.72,-1) to [out=60,in=0] (.2,.8) to [out=-180,in=120] (-.5,.5) to [out=-60,in=-120] (1.72,-1);
      \draw [thick,->] (.015,.8) -- (.025,.8);
    \end{knot}
    \end{tikzpicture}
    \right]_{\beta=2}
\]

Then one can slide the curl in the first term, and again apply the skein relation on a crossing in the second term, to obtain the following expression:
\[
\left[
    \begin{tikzpicture}[anchorbase,scale=.5]
  \begin{knot}[
            clip width=7,
            consider self intersections,
            end tolerance=1pt,
            flip crossing=3,
            flip crossing=2,
    ]
     \strand [thick] (0,0) to [out=0,in=60] (.86,-.5) to [out=-120,in=-60] (-1.72,-1) to [out=120,in=180] (0,1) to [out=0,in=60] (1.72,-1) to  [out=-120,in=-60] (-.86,-.5) to [out=120, in=180] (0,0);
     \draw [thick,->] (-.05,0) -- (.05,0);
   \end{knot}
    \end{tikzpicture}
    \right]_{\beta=2}
+(q^{-1}-q)
\left[
    \begin{tikzpicture}[anchorbase,scale=.5]
  \begin{knot}[
            clip width=7,
            consider self intersections,
            end tolerance=1pt,
            flip crossing=1,
            flip crossing=2,
    ]
    \strand [thick] (0,2) to [out=180,in=60] (-.86,.6) to [out=-120, in=120] (-1.72,-1) to [out=-60, in=-120] (.86,.5) to [out=60,in=0] (0,2);
    \draw [thick,->] (-.05,2) -- (.05,2);
     \strand [thick] (1.72,-1) to [out=60,in=0] (.2,.8) to [out=-180,in=120] (-.5,.5) to [out=-60,in=-120] (1.72,-1);
     \draw [thick,->] (.015,.8) -- (.025,.8);
   \end{knot}
    \end{tikzpicture}
    \right]_{\beta=2}
+(q^{-1}-q)^2
\left[
    \begin{tikzpicture}[anchorbase,scale=.5]
  \begin{knot}[
            clip width=7,
            consider self intersections,
            end tolerance=1pt,
            flip crossing=1,
            flip crossing=2,
    ]
    \strand [thick] (0,2) to [out=180,in=60] (-.86,.6) to [out=-120, in=120] (-1.72,-1) to [out=-60, in=-90] (.7,.4) to [out=90,in=0] (.2,.8) to [out=-180,in=120] (-.5,.5) to [out=-60,in=-120] (1.72,-1) to [out=60,in=-45] (.9,.6) to [out=135,in=0] (0,2);  
    \draw [thick,->] (-.05,2) -- (.05,2);
  \end{knot}
    \end{tikzpicture}
    \right]_{\beta=2}
\]

From there everything evaluates thanks to Equations~\eqref{eq:54} that allow to remove the curls, before one evaluates the circles. One gets:
\[
q^{-2}[2]+(q^{-1}-q)[2]^2 +q^{-2}(q^{-1}-q)^2[2]=[2](q^{-4}+1-q^2)
\]

This equals $-q^3+q^{-1}+q^{-3}+q^{-5}$ as expected.
\end{example}

\begin{remark}
  It is an interesting question to wonder what happens if one sets $d<0$.
\end{remark}

\begin{remark}\label{rem:tangleBrRT}
More generally, one can define an invariant associated to tangles: a tangle $T$ gets assigned $[T]\in \Hom_{\uBr(\beta)}(\partial T^{-},\partial T^{+})$. For generic values of $\beta$, this can be seen as a HOMFLY-PT invariant for $T$, but it is interesting to note that the $\beta=d$ specialization does \emph{not} exactly match the $\slnn{d}$ Reshetikhin-Turaev invariant. We will come back to that fact later.
\end{remark}

It is a classical fact that the HOMFLY-PT polynomial admits a specialization to both the Reshetikhin-Turaev invariants and the Alexander polynomial, which is easily checked by comparing Conway-type skein relations for all invariants. We have seen the former, but for the latter case one needs to introduce reduced versions of the invariants. Indeed, the specialization is supposed to use $\beta=0$, which, with the current definition, would systematically yield $[L]=0\in \End_{\uBr(0)}(\emptyset)$ for any non-empty link (recall that $[\beta]=0$ if $\beta =0$).

To go to the reduced case, notice that another consequence of Proposition \ref{prop:DDS} is that $\End_{\uBr(\beta)}(\uparrow)\simeq \fieldk$. This is a perfect situation to produce polynomial invariants and motivates the following definition.

\begin{definition}
  Let $L$ be a link, and $\tilde{L}$ the $(1,1)$ tangle obtained by cutting open one strand. Then:
  \begin{itemize}
  \item if $\beta$ is generic, $[\tilde{L}]\in \End_{\uBr(\beta)}(\uparrow)\simeq \C(q,q^{\beta})$ is the reduced HOMFLY-PT polynomial of $L$;
  \item if $\beta=n>0$, $[\tilde{L}]\in \End_{\uBr(\beta=n)}(\uparrow)\simeq \C(q)$ is the reduced $\sln$ Reshetikhin-Turaev invariant of $L$;
  \item if $\beta=0$, $[\tilde{L}]\in \End_{\uBr(\beta=0)}(\uparrow)\simeq \C(q)$ is the Alexander polynomial of $L$.
  \end{itemize}
\end{definition}

This is all well-defined, thanks to the next proposition.

\begin{proposition}
With the above notations, the value $[\tilde{L}]$ is independent on the choice of the cutting place.
\end{proposition}

\begin{proof}
  For $\beta$ generic,
  \[
    [\tilde{L}]_\beta=\frac{[L]_\beta}{\left[
      \begin{tikzpicture}[anchorbase,scale=.4]
        \draw [thick] (0,0) circle (1);
      \end{tikzpicture}
    \right]_\beta}
    \]
    which can be seen by reclosing the $(1,1)$ tangle. In the expression above, the denominator is non-vanishing, and thus the value of $[\tilde{L}]$ does not depend on a specific choice of $\tilde{L}$ for a given $L$.

For other values of $\beta$, the result follows by specialization.
\end{proof}

\begin{remark}
  The global approach permitted by the Brauer category makes the above proof very easy. If one were to only consider the Alexander polynomial on its own, this proof would become somewhat trickier (see \cite{SarAlexander} for example).
  \end{remark}

\subsection{The scalar principle}

Consider the situation, schematized below, of a functor from the tangle category to some $\fieldk$-linear category $\mathcal{C}$ factoring through the Brauer category:
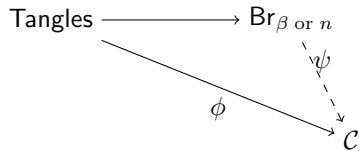
\[
\begin{tikzpicture}[anchorbase, scale=.8]
  \node  (A) at (0,0) {$\catTangles$};
  \node (B) at (4,0) {$\uBr_{\beta\; \text{or} \; n}$};
  \node (C) at (5,-2) {$\mathcal{C}$};
  \draw [->] (A)  -- (C) node[midway,below]{$\phi$};
  \draw [->] (A) -- (B);
  \draw [dashed,->] (B) -- (C) node[midway, above] {$\psi$};
\end{tikzpicture}
\]
Then:
\begin{itemize}
\item $\phi(L)$ is the HOMFLY-PT polynomial if $\beta$ is generic and the $\sln$ Reshetikhin-Turaev invariant for $\beta=n>0$;
  \item $\phi(\tilde{L})$ is the reduced HOMFLY-PT polynomial if $\beta$ is generic, the reduced $\sln$ Reshetikhin-Turaev invariant if $\beta=n>0$, and the Alexander polynomial if $\beta=0$.
\end{itemize}

Note however that this principle is specific to links, and that there could (and will) be more room for novelty in the tangle case.

\section{A representation-theoretic functor}\label{sec:RT}

In this section, we will define families of functors from $\uBr(\beta=d)$ to categories of representations of some quantum groups, that is, categories of vectors spaces with extra structure. This will naturally produce tangle invariants.

\subsection{Definition of $U_q(\gl_{m|n})$}

Quantum groups arise as deformations of enveloping algebras of some Lie algebras (see \cite{Lus4,ChariPressley,Kassel} for complete accounts on the theory). In the case of $\sln$, they can be thought of as $q$-deformed versions of Lie algebras originating from matrix spaces, and the action on the vector space is still very present in the theory. In the super case we are going to consider, instead of starting with endomorphisms of a vector space, one starts with endomorphisms of a super-vector space, that is, a $\Z/2\Z$-graded vector space.

Following this general idea, for a pair of non-negative integers $(m,n)$, we define a degree function:
\begin{align*}
  \{1,\cdots, m+n\}&\mapsto \Z/2\Z \\
               i&\mapsto |i|
\end{align*}
that assigns $0$ to the first $m$ entries and $1$ to the last $n$ entries. We take the following definition, which is similar to the ones from \cite{Scheunert,Zhang_super_qgroup}.

\begin{definition}
The \emph{quantum enveloping superalgebra} $U_q(\gl_{m|n})$ is defined to be the unital superalgebra
over $\C (q)$ with generators $E_i$, $F_i$ for $i\in \{1,\dots,m+n-1\}$ and $L_i$  for $i \in \{1,\cdots, m+n\}$
subject to the following relations (we introduce $K_i=L_i^{(-1)^{|i|}}L_{i+1}^{-(-1)^{|i+1|}}$): 
\vspace{\abovedisplayskip}
\begin{subequations}
\begingroup
\setlength{\belowdisplayskip}{\jot}
\setlength{\belowdisplayshortskip}{\jot} 
\setlength{\abovedisplayskip}{0pt}
\setlength{\abovedisplayshortskip}{0pt}
  \begin{align}
    L_i E_i &=qE_i L_i, & L_i F_i&=q^{-1}F_iL_i\\
        L_{i} E_{i+1} &=q^{-1}E_{i+1} L_i, & L_i F_{i+1}&=qF_{i+1}L_i
\label{eq:13}
\end{align}
\begin{align}(-1)^{\abs{i}} E_iF_i - (-1)^{\abs{i+1}} F_iE_i =  \frac{K_i-K_i^{-1}}{q-q^{-1}}\label{eq:14}
\end{align}
\begin{align}
E_m^2 & =F_m^2=0\label{eq:15}\\
E_iE_j=E_jE_i &\text{ and } F_iF_j = F_jF_i &&\text{if } \abs{i-j}\geq 2\label{eq:16}\\
E_iF_j&=F_jE_i &&\text{if } i \neq j\label{eq:17}
\end{align}
\begin{align}
E_i^2 E_{i+1} - [2] E_i E_{i+1} E_i + E_{i+1} E_i^2 &= 0 &\text{ for } i,i+1 \neq m \label{eq:18}\\
E_{i+1}^2 E_{i} - [2] E_{i+1} E_{i} E_{i+1} + E_{i} E_{i+1}^2& = 0 &\text{ for } i,i+1 \neq m\label{eq:19}\\
F_i^2 F_{i+1} - [2] F_i F_{i+1} F_i + F_{i+1} F_i^2 &= 0 &\text{ for } i,i+1 \neq m\label{eq:20}\\
F_{i+1}^2 F_{i} - [2] F_{i+1} F_{i} F_{i+1} + F_{i} F_{i+1}^2 &= 0 &\text{ for } i,i+1 \neq m\label{eq:21}
\end{align}
\begin{gather}
\begin{multlined}
  E_m E_{m-1} E_m E_{m+1} + E_{m-1} E_m E_{m+1} E_m + E_m
  E_{m+1} E_m E_{m-1} \\+ E_{m+1} E_m E_{m-1} E_m - [2] E_m
  E_{m-1} E_{m+1} E_m = 0
\end{multlined}
\label{eq:22}\\
\begin{multlined}
  F_m F_{m-1} F_m F_{m+1} + F_{m-1} F_m F_{m+1} F_m + F_m
  F_{m+1} F_m F_{m-1} \\+ F_{m+1} F_m F_{m-1} F_m - [2] F_m
  F_{m-1} F_{m+1} F_m = 0
\end{multlined}
\label{eq:23}
\end{gather}
\endgroup
\end{subequations}
\vspace{\belowdisplayskip}
\end{definition}

Note that the non-super case is recovered by taking $n=0$, in which case Relation~\eqref{eq:14} specializes to the usual relation:
\[
E_iF_i-F_iE_i=\frac{K_i-K_i^{-1}}{q-q^{-1}}
\]
The complicated equations~\eqref{eq:22} and \eqref{eq:23} disappear.

We define a \emph{comultiplication} $\Delta\colon U_q \mapto U_q \otimes U_q$, a \emph{counit} $\counit\colon U_q \mapto \C(q)$ and an \emph{antipode} $S\colon U_q \mapto U_q$ by setting on the generators:
\begin{equation}
\begin{aligned}
  \Delta(E_i)&= E_i \otimes K_i^{-1}+1 \otimes E_i, & \Delta(F_i)&=F_i \otimes 1 + K_i \otimes F_i\\
  S(E_i)&=-E_iK_i, & S(F_i)&=- K_i^{-1}F_i\\
  \Delta(L_i)&=L_i \otimes L_i, &   S(L_i)&=L_i^{-1}\\
  \counit(E_i)& =\counit(F_i)=0, & \counit(L_i)&=1
\end{aligned}\label{eq:24}
\end{equation}
and extending $\Delta$ and $\counit$ to superalgebra homomorphisms and $S$ to a superalgebra anti-homomorphism.

These maps endow the algebra with the structure of a Hopf algebra. The precise axioms of Hopf algebras won't play a front stage role in what follows, but they could be summarized by saying that they allow to define the notion of tensor product of representations and the notion of the dual of a representation: given $V$ and $W$ representations of $U_q(\gl_{m|n})$, one can define $U_q(\gl_{m|n})$ actions on $V\otimes W$ and $V^{\ast}$ by letting:
\begin{gather}
  x\in U_q(\gl_{m|n})\;\text{acts on}\; V\otimes W\;\text{by}\;\Delta(x),\\
  \text{and on}\; f\in V^{\ast}\text{via}\;x(f)(v)=f(S(x)v),\;\forall v\in V.
\end{gather}

Just as for the usual general or special linear groups of matrices, one can define a subalgebra $U_q(\slnn{m|n})$ which will sometimes be more handy for topological applications.

\begin{definition} \label{def:slmn}
  The quantum enveloping superalgebra $U_q(\slnn{m|n})$ is the subalgebra of $U_q(\gl_{m|n})$ generated by the $E_i$'s, $F_i$'s and $K_i$'s.
\end{definition}

\subsection{Representations}

As stated earlier, the main idea at play here is that $U_q(\gl_{m|n})$ is a deformed version of $\End(\C^{m|n})$. A $q$-deformed version of $\C^{m|n}$ will thus be central in the story.

Denoting $\C_q:=\C(q)$, let $V:=\C_q^{m|n}$, that is, an $m+n$ dimensional vector space with distinguished basis vectors $x_1,\dots, x_m$ of $\Z/2\Z$-degree zero, and $x_{m+1},\dots,x_{m+n}$ of degree one. An action of $U_q(\gl_{m|n})$ on $V$ can be described by:
\begin{gather}
  E_i x_j=\delta_{i+1,j}x_i,\quad F_i x_j=\delta_{i,j}x_{i+1} \\
  L_ix_i=q x_i,\quad L_jx_i=x_i \;\text{if}\;j\neq i
\end{gather}

From there, one can deduce:
\begin{gather}
  K_iv_i=q v_i\; \text{and}\; K_iv_{i+1}=q^{-1}v_{i+1} \; \text{for}\; i<m\\
  K_mv_m=q v_m\; \text{and}\; K_mv_{m+1}=qv_{m+1}\\
  K_iv_i=q^{-1} v_i\; \text{and}\; K_iv_{i+1}=qv_{i+1} \; \text{for}\; i>m\\
\end{gather}
The middle line is of course key in understanding the difference between the usual and super cases.

Using the definition of the action of the dual representation, one can also derive:
\begin{align}
  E_ix_i^\ast&= -q^{-1}x_{i+1}^\ast\;\text{if}\; i<m \\
  E_ix_i^\ast&= -qx_{i+1}^\ast\;\text{if}\; i\geq m \\
  K_ix_i^\ast&= q^{-1}x_{i}^\ast\;\text{if}\; i\leq m \\
  K_ix_i^\ast&= qx_{i}^\ast\;\text{if}\; i> m \\
  K_ix_{i+1}^\ast&= qx_{i+1}^\ast\;\text{if}\; i<m \\
  K_ix_{i+1}^\ast&= q^{-1}x_{i+1}^\ast\;\text{if}\; i\geq m 
\end{align}

These basic pieces can be used to produce several tensor categories:
\begin{itemize}
\item $\catRep(U_q(\gl_{m|n}))$: the monoidal category generated by $V$ and $V^\ast$;
\item $\catRep^{\uparrow}(U_q(\gl_{m|n}))$: the monoidal category generated by $V$.
\end{itemize}

In the case when $n=0$, these two categories are actually very close: they become equivalent once one restricts to $U_q(\slnn{m})$.

\subsection{The Reshetikhin-Turaev functor}

Reshetikhin and Turaev defined a functor from the category of tangles to the category of representations of $U_q(\sln)$, that in the $\slnn{2}$ case encompasses the Jones polynomial. The invariants of links thus defined were then extended to 3-manifolds after specialization at roots of unity, yielding the so-called Witten-Reshetikhin-Turaev invariant \cite{Witten3,RT}. In the super case, the definition of the functor can be traced back to Zhang~\cite{Zhang} and Geer-Patureau~\cite{GP_supertrace}, the extension of which to 3-manifolds is due to Blanchet, Costantino, Geer and Patureau-Mirand~\cite{CGP, BCGP}.

All of these functors factor through the Brauer category with $\beta=m-n$, yielding a unified presentation for them all. Most sources give just enough of the definition so that it can be extended to all generators. Here I've tried to list a little bit more, hoping it can be useful to someone wishing to do explicit computations.\footnote{I also hope by this process to suppress one recurring issue: running these computations on a small piece of paper, storing it somewhere, forgetting where exactly, and having to run the same computations again a year later for the next project.}

On objects, the functor sends $\uparrow$ to $V$, and its dual $\downarrow$ to $V^{\ast}$.

On morphisms, one sends cups to the following maps:
\begin{align}
\begin{tikzpicture}[anchorbase,scale=.5]
    \draw [->] (1,1) -- (1,.5) to [out=-90,in=-90] (0,.5) -- (0,1);
\end{tikzpicture}
\longrightarrow
\quad
& \C_q\rightarrow V\otimes V^{\ast} \nonumber \\
& 1\rightarrow \sum_{k=1}^{m+n}x_k\otimes x_{k}^\ast
\\
\begin{tikzpicture}[anchorbase,scale=.5]
    \draw [<-] (1,1) -- (1,.5) to [out=-90,in=-90] (0,.5) -- (0,1);
\end{tikzpicture}
\longrightarrow
\quad
& \C_q\rightarrow V\otimes V^{\ast}  \nonumber \\
&  1 \rightarrow q^{m-n}(\sum_{k=1}^{m}q^{1-2k} x_k^\ast\otimes x_{k}-\sum_{k=m+1}^{m+n}q^{1-2k-4m}x_k^{\ast}\otimes x_k)
\end{align}
and caps are handled as follows:
\begin{align}
\begin{tikzpicture}[anchorbase,scale=.5]
    \draw [->] (1,-1) -- (1,-.5) to [out=90,in=90] (0,-.5) -- (0,-1);
\end{tikzpicture}
\longrightarrow
\quad
&V\otimes V^\ast \rightarrow \C_q \nonumber \\
& x_k\otimes x_{k}^\ast \rightarrow 1
\\
\begin{tikzpicture}[anchorbase,scale=.5]
    \draw [<-] (1,-1) -- (1,-.5) to [out=90,in=90] (0,-.5) -- (0,-1);
\end{tikzpicture}
\longrightarrow
\quad
&V\otimes V^\ast \rightarrow \C_q \nonumber \\
& x_k\otimes x_{k}^\ast \rightarrow q^{-m+n-1+2k}\;\text{for}\;k\leq m\\
& x_k\otimes x_{k}^\ast \rightarrow -q^{3m+n+1-2k}\;\text{for}\;k>m
\end{align}

The doubtful reader might want to check that clockwise and counter-clockwise oriented circles do get sent to $[m-n]$.

Finally, the crossings get sent to the so-called quantum $R$-matrix:
\begin{align}
     \begin{tikzpicture}[smallnodes,anchorbase,xscale=0.7,yscale=0.5]
   \draw[uno] (1,0)  -- ++(0,0.3) \midarrow .. controls ++(0,0.7) and ++(0,-0.7) .. ++(-1,1.4) -- ++(0,0.3)  \midarrow ;
   \draw[uno,cross line] (0,0)  -- ++(0,0.3) \midarrow .. controls ++(0,0.7) and ++(0,-0.7) .. ++(1,1.4) -- ++(0,0.3)  \midarrow ;
     \end{tikzpicture}
     \longrightarrow \quad & V\otimes V\rightarrow V\otimes V \\
     & x_i\otimes x_j \rightarrow
       \begin{cases}
    q^{-1} x_i \otimes x_i & \text{if } i = j \leq m\\
    (-1)^{\abs{i}\abs{j}} x_j \otimes x_i & \text{if } i < j\\
    (-1)^{\abs{i}\abs{j}} x_j \otimes x_i + (q^{-1}- q) x_i \otimes x_j & \text{if } i > j\\
    -q x_i \otimes x_i & \text{if } i = j > m
  \end{cases}
\end{align}

The following theorem summarizes works of Reshetikhin-Turaev~\cite{RT}, Geer-Patureau~\cite{GP_supertrace} and Zhang~\cite{Zhang}. It can be directly proved by explicit computation using the previous formulas.

\begin{theorem}
  The above map induces a functor:
  \[
\uBr(\beta)\mapsto \catRep(U_q(\gl_{m|n}))
  \]
\end{theorem}

One thus extracts tangle invariants from there.

\begin{definition}
  The $\gl_{m|n}$ Reshetikhin-Turaev invariant of a tangle is its image under the composition:
  \[
\catTangles\mapsto \uBr(\beta=m-n)\mapsto \catRep(U_q(\gl_{m|n}))
  \]
\end{definition}

\begin{remark}
  For $K$ a knot or link, the invariant associated to $K$ lives in
  \[
\End_{\catRep(U_q(\gl_{m|n}))}(\C_q)\simeq \C_q
\]
and can be shown to be a Laurent polynomial. On the other hand, for a more general tangle, the morphism space that contains the invariant is not $1$-dimensional anymore. Being able to perform actual computations in these spaces is a bit of a challenge.
  \end{remark}

The following corollary recovers a well-known phenomenon.

\begin{corollary}[of the scalar principle]
  The $\gl_{m|n}$ Reshetikhin invariant of a link only depends on $d=m-n$.
\end{corollary}

This also shows that the definition above agrees with the one presented in Definition~\ref{def:1} in the case of $\gl_{m|0}$.

Let us now go back to Remark~\ref{rem:tangleBrRT} and comment on the difference between the invariant associated to a tangle in the Brauer category and its representation-theoretic version.

One can show (for example using ideas from Schur--Weyl duality, which will be the focus of the next section) that the following situation:
\[
\begin{tikzpicture}[anchorbase]
  \node (A) at (0,0) {$\uBr(\beta=m-n)$};
  \node (B) at (5,1) {$\catRep(U_q(\gl_{m|n})$};
  \node (C) at (5,-1) {$\catRep(U_q(\gl_{m+1|n+1})$};
  \draw [->] (A)  -- (B) node[midway,below]{$\phi_{m|n}$};
  \draw [->] (A)  -- (C) node[midway,below]{$\phi_{m+1|n+1}$};  
\end{tikzpicture}
\]
yields the following inclusions: $ker(\phi_{m+1|n+1})\subset ker(\phi_{m|n})$, but the kernels can differ. However, if one fixes the number of strands, they stabilize for large values of $m$ and $n$, and $\uBr(\beta=d)$ can thus be thought of as the ``inverse limit'' of the $\gl_{m|n}$ invariants for fixed values of $m-n=d$.

\section{Schur--Weyl and skew Howe dualities}\label{sec:SW}

The observation motivating this section is that it is a hard task to perform computations in $\catRep(U_q(\gl_{m|n}))$. On the other hand, $\uBr(\beta)$ is somewhat easier to manipulate, but one gives up some of the algebraic aspects by going to this more pictorial, topological version. So we want to find some algebraic object, with an easy definition, that will give a firmer hand on the maps that are used to define the Reshetikhin-Turaev invariant. Furthermore, we will see that the proposed solution to that problem has the nice feature of extending to the generic $\beta$ case as well, yielding a united, quantum-group-based definition of the HOMFLY-PT polynomial together with its specialization.

To make the definition simpler, we will restrict to braids and their closures. This is not strictly necessary, but will allow us to use much lighter notations. Two kinds of dualities, not unrelated to one another, are at play: Schur--Weyl duality and skew Howe duality. The following picture (where $B$ stands for a braid) aims at giving a quick idea about their appearance in our situation.

\[
\begin{tikzpicture}[anchorbase,scale=.6]
  \node at (0,0) {$B$};
  \draw (-1,-1) rectangle (1,1);
  \draw [->] (.8,1) -- (.8,1.5) to [out=90,in=90] (1.7,1.5) -- (1.7,-1.5) to [out=-90, in=-90] (.8,-1.5) -- (.8,-1);
  \draw [->] (-.8,1) -- (-.8,1.5) to [out=90,in=90] (3.3,1.5) -- (3.3,-1.5) to [out=-90, in=-90] (-.8,-1.5) -- (-.8,-1);
  \node at (0,1.3) {$\cdots$};
  \draw [decorate,decoration={brace,amplitude=6pt},xshift=-2pt,yshift=0pt] (-1,-1) -- (-1,1)node [black,midway,xshift=-30pt] {\small Schur--Weyl};
  \draw [decorate,decoration={brace,amplitude=6pt},xshift=-2pt,yshift=0pt] (3.5,2.5) -- (3.5,-2.5)node [black,midway,xshift=30pt] {\small skew Howe};
\end{tikzpicture}
\]

\subsection{Hecke algebra}

\begin{definition}
  We let $\calH_N:=\End_{\uBr(\beta=d)}(\uparrow^{\otimes N})$.
\end{definition}

It appears that the definition above does not depend on $d$ (this mostly follows from Proposition~\ref{prop:DDS}). The reason to go through the trouble of choosing $d$ is that in the case where $\beta$ is generic, the base field is larger than the usual $\C_q$ we are looking for. Setting $\beta=d$ allows to recover the desired base field.

An important feature, which holds as well for generic values of $\beta$, is that $\End_{\uBr(\beta)}$ is spanned by braid diagrams. Equivalently (and perhaps in a more down-to-earth manner), one could consider in $\uBr(\beta)$ the $\C_q$-span of braid diagrams in the same endomorphism space.

The choice of braids as spanning sets helps recovering a classical presentation for the Hecke algebra. If one denotes:
\[
T_i=\quad
 \begin{tikzpicture}[anchorbase,scale=.5]
  \begin{knot}[
      clip width=7,
      consider self intersections,
      end tolerance=1pt,
      flip crossing=,
    ]
    \draw [thick] (-1.5,0) -- (-1.5,1);
    \node at (-.75,.5) {$\cdots$};
    \strand [thick] (0,0) to [out=90,in=-90] (1,1);
    \strand [thick] (1,0) to [out=90,in=-90] (0,1);
    \node at (1.75,.5) {$\cdots$};
    \draw [thick] (2.5,0) -- (2.5,1);
    \node at (-1.5,-.3) {\tiny $1$};
    \node at (0,-.3) {\tiny $i$};
    \node at (1,-.3) {\tiny $i$+$1$};
    \node at (2.5,-.3) {\tiny $N$};
  \end{knot}
\end{tikzpicture}
 \in \calH_N
\]
then $\calH_N$ can be presented by the generators $T_i$'s with relations:
\begin{gather}
  T_iT_{i+1}T_i=T_{i+1}T_{i}T_{i+1},\;1\leq i\leq N-2\\
  (T_i-q^{-1})(T_i+q)=0,\;1\leq i\leq N-1\\
  T_iT_j=T_jT_i,\; |i-j|\geq 2
\end{gather}

This recovers (a version of) the usual presentation for $\calH_N$. Just like braid groups surject onto symmetric groups, there is a surjection from Hecke algebras to group algebras of symmetric groups as follows.

\begin{remark}
  There is a surjective map:
  \begin{align}
    \calH_N &\twoheadrightarrow \C[\mathfrak{S}_N] \\
    T_i &\rightarrow (i,i+1)
  \end{align}
\end{remark}

\subsection{Schur--Weyl duality}

We have seen:

\[
\begin{tikzpicture}
  \node (A) at (0,0) {$\calH_N$};
  \node (B) at (4,0) {$\End_{U_q (\gl_{m|n})}$};
  \node (C) at (0,-3) {$\End_{\uBr(\beta=m-n)}(\uparrow^{\otimes N})$};
  \draw [->] (A) -- (B);
  \draw [right hook->] (A) -- (C);
  \draw [->] (C) -- (B);
\end{tikzpicture}
\]

This can be rephrased by saying that there are two commuting actions of $\calH_N$ and $U_q(\gl_{m|n})$ on $V^{\otimes N}$:
\[
\mathcal{H}_N\quad
\begin{tikzpicture}[scale=.5,anchorbase]
  \draw [->] (0,-.2) to [out=-120,in=-90] (-1,0) to [out=90,in=120] (0,.2);
\end{tikzpicture}
\quad V^{\otimes N}
\quad
\begin{tikzpicture}[scale=.5,anchorbase]
  \draw [->] (0,-.2) to [out=-60,in=-90] (1,0) to [out=90,in=60] (0,.2);
\end{tikzpicture}
\quad
U_q(\gl_{m|n})
\]

The classical version of this situation looks like:
\[
\mathfrak{S}_N\quad
\begin{tikzpicture}[scale=.5,anchorbase]
  \draw [->] (0,-.2) to [out=-120,in=-90] (-1,0) to [out=90,in=120] (0,.2);
\end{tikzpicture}
\quad (\C^m)^{\otimes N}
\quad
\begin{tikzpicture}[scale=.5,anchorbase]
  \draw [->] (0,-.2) to [out=-60,in=-90] (1,0) to [out=90,in=60] (0,.2);
\end{tikzpicture}
\quad
GL_m
\]
for which Schur--Weyl duality asserts that not only these two actions commute, but they generate each other commutant. In the quantum version, an analogous statement holds, as stated below.

\begin{theorem}
  The actions of $\calH_N$ and $U_q(\gl_{m|n})$ on $V^{\otimes N}$ commute with each other and generate each other centralizers.
\end{theorem}
We refer to~\cite{Etingof,Jimbo,Mitsuhashi} for more precise versions and more details.

In our context, what we want to retain from this is the existence of the following map:
\[
\calH_N\twoheadrightarrow \End_{U_q(\gl_{m|n})}(V^{\otimes N})
\]
Furthermore, it can be shown (and this is part of the more detailed statements of Schur--Weyl duality) that if $m\geq N$, this map is an isomorphism (for a more precise statement in the super case, see \cite[Theorem 5.1]{Mitsuhashi}).

This thus gives us an algebraic description of $\End_{U_q(\gl_{m|n})}(V^{\otimes N})$, in which the Reshetikhin-Turaev invariant lives: the Reshetikhin-Turaev map that to a braid assigns an intertwiner factors through $\calH_N$. However, in this version, only braids show up, which more or less means that we are lacking a notion of duality for the objects.

\subsection{Skew-Howe duality}

Here is the picture we are aiming at:

\[
\begin{tikzpicture}[anchorbase]
  \node at (-3.5,0) {$U_q(\gl_{m|n})$};
  \node at (-2,0) {
    $\begin{tikzpicture}[scale=.5,anchorbase]
  \draw [->] (0,-.2) to [out=-120,in=-90] (-1,0) to [out=90,in=120] (0,.2);
    \end{tikzpicture}$};
  \node at (0,0) {${\bigwedge_q}^N(\C_q^{m|n}\otimes \C_q^{k})$};
  \node at (2,0) {$
\begin{tikzpicture}[scale=.5,anchorbase]
  \draw [->] (0,-.2) to [out=-60,in=-90] (1,0) to [out=90,in=60] (0,.2);
\end{tikzpicture}
$};
\node at (3.5,0) {$
  U_q(\gl_{k})
  $};
\node [rotate=-90] at (0,-1) {$\simeq$};
\node at (0,-2) {$\bigoplus_{a_1+\cdots +a_k=N} {\bigwedge_q}^{a_1}(\C_q^{m|n})\otimes \cdots \otimes {\bigwedge_q}^{a_k}(\C_q^{m|n})$};
\end{tikzpicture}
\]

Note that we could have considered a slightly more general situation, allowing the quantum group on the right-hand side above to also be super (see~\cite[Theorem 4.2 and Proposition 4.3]{QS2}). The above version appears in~\cite[Theorem 2.2]{WuZ}.

From there one gets:
\[
U_q(\gl_{k})\twoheadrightarrow \End_{U_q(\gl_{m|n})}\left({\bigwedge}_q^N(\C_q^{m|n}\otimes \C_q^k)\right)
\]
The base space on the rhs above (or on the bottom line of the previous picture) contains tensor products of exterior powers of the vector representation, which, in the non-super case, encompasses both $V$ and $V^{\ast}$ (provided one goes from $\gl$ to $\slnn{}$). In other words, when $n=0$ the endomorphism space hit by $U_q(\gl_k)$ contains all we need to compute the Reshetikhin-Turaev invariant for any tangle.

Before going into more detail about the precise relationship with Reshetikhin-Turaev invariants, let us now explain the different pieces in the story, starting with quantum exterior power.

Recall that $V=\C_q^{m|n}$ has basis $\{x_i\}_{1\leq i\leq m+n}$. Define $I\subset V\otimes V$ to be the vector space spanned by:
\[\{x_i\otimes x_i,\;i\leq m\}\cup \{x_i\otimes x_j+(-1)^{|i||j|}q^{-1}x_j\otimes x_i,\;i<j\}\]
$I$ is a subrepresentation of $V^{\otimes 2}$, and one can thus define a quotient representation as follows.

\begin{definition}
One defines ${\bigwedge}_q^l(V)$ as follows.
  \[
{\bigwedge}_q^l(V):=V^{\otimes l}/\langle V^{\otimes s}\otimes I\otimes V^{l-s-2},\; 0\leq s\leq l-2\rangle
  \]
\end{definition}

This defines the notion of quantum exterior power, which $q$-deforms the classical exterior power of a representation.

It can be shown that this yields an irreducible representation of $U_q(\gl_{m|n})$, and by construction it appears as a subrepresentation of $V^{\otimes l}$, meaning that there exist $\iota$ and $\pi$:
\[
V^{\otimes l}\xrightarrow{\pi}{\bigwedge}_q^l(V)\xrightarrow{\iota} V^{\otimes l}
\]
with $p_{\wedge}:=\iota\circ \pi$ an idempotent. An important point is that $p_{\wedge}$ is realized by an element of the Hecke algebra. The formula for it appears for example in~\cite[p. 250]{Jimbo}:
\[
\frac{1}{[l]!}\sum_{w\in \mathfrak{S}_l}(-1)^{\ell(w)}q^{-\frac{l(l-1)}{2}+\ell(w)}T_w
\]
where $\ell(w)$ is the length of the word $w\in \mathfrak{S}_N$.

We can now come to the statement of quantum skew Howe duality, referring to~\cite{Howe,BerensteinZwicknagl,CKM,ChengWang,WuZ,QS2} for more details:
\begin{theorem}
  In the following situation:
  \[
U_q(\gl_{m|n}) \quad
\begin{tikzpicture}[scale=.5,anchorbase]
  \draw [->] (0,-.2) to [out=-120,in=-90] (-1,0) to [out=90,in=120] (0,.2);
\end{tikzpicture}
\quad {\bigwedge}_q^N(\C_q^{m|n}\otimes \C_q^{k|l})
\quad
\begin{tikzpicture}[scale=.5,anchorbase]
  \draw [->] (0,-.2) to [out=-60,in=-90] (1,0) to [out=90,in=60] (0,.2);
\end{tikzpicture}
\quad
U_q(\gl_{k|l})
\]
the two actions commute and generate each other commutant.
\end{theorem}

Again, we'll mostly be interested in the following fact, stated in the non-super case:
\[
\phi:U_q(\gl_k)\twoheadrightarrow \End_{U_q(\gl_m)}\left({\bigwedge}_q^N(\C_q^m\otimes \C_q^k)\right)
\]
The map is not faithful, but it stabilizes when $m$ grows, and one can define the quantum Schur algebra from it.

\begin{definition}
  The quantum Schur algebra $\calS_q(k,N)$ is defined as:
  \[
\calS_q(k,N):=U_q(\gl_k)/ker(\phi)\quad \text{for}\; m\geq N
  \]
  \end{definition}

The non-super statement of quantum skew Howe duality was used by Cautis-Kamnitzer-Morrison to give a presentation by generators and relations of the category $\catRep^{+}(U_q(\gl_{m}))$, the category monoidally generated by the exterior powers of the vector representations. In the super case, such presentations are in general yet to be found. The $\gl{(1|1)}$ case can be found in \cite{Grant,Sartori_cat}.

Let's now go back to the space ${\bigwedge}_q^N(\C_q^m\otimes \C_q^k)$ and present it in a version that looks closer to the tensor product of vector representations that we care about. To do so, one can show that the quantum exterior power behaves much like the usual exterior power with respect to direct sum, and decompose $\C_q^k\simeq \C_q\oplus \cdots \oplus \C_q$, so that one gets:
\begin{equation}\label{eq:wedge+}
  {\bigwedge}_q^N(\C_q^{m|n}\otimes \C_q^k)\simeq
  \bigoplus_{a_1+\cdots +a_k=N}{\bigwedge}_q^{a_1}(\C_q^{m|n})\otimes \cdots \otimes {\bigwedge}_q^{a_k}(\C_q^{m|n})
\end{equation}

Let us remark for later use that ${\bigwedge}_q^{a}(\C_q^{m|n})$ is $\{0\}$ if $a<0$, and if $n=0$ then ${\bigwedge}_q^{q}(\C_q^m)$ is also $\{0\}$ if $a>m$. Similarly, ${\bigwedge}_q^0(\C_q^{m|n})\simeq \C_q$ and for $n=0$ ${\bigwedge}_q^m(\C_q^m)$ is the determinant representation, also of dimension 1.

We have already mentioned the idea of restricting from $\gl_m$ to $\slnn{m}$, and it will again play a role soon, so let us give a quick explanation. As appears from Definition~\ref{def:slmn}, $U_q(\slnn{m})$ is a subalgebra of $U_q(\gl_m)$. So any representation $V$ of $U_q(\gl_m)$ is also a representation of $U_q(\slnn{m})$ by what's called restriction, which is simply the composition: $U_q(\slnn{m})\hookrightarrow U_q(\gl_m) \rightarrow \End(V)$. Of course, one looses some information when restricting the action to a subalgebra, and for example one can explicitly check that ${\bigwedge}_q^0(\C_q^m)$ and ${\bigwedge}_q^m(\C_q^m)$ both become isomorphic, trivial representations once restricted to $U_q(\slnn{m})$ (the only difference lied in the action of the $L_i$'s, but these differences balance when one forms $K_i=L_{i}L_{i+1}^{-1}$).

\subsection{Schur--Weyl and skew Howe dualities}

Aaron Lauda mentioned to me a few years ago that one could actually realize Schur--Weyl duality as an instance of skew Howe duality, and vice-versa. I thought these notes might be a good opportunity to try to get a clearer picture of this. I am not 100\% sure that the result is as natural as I hoped, but let us give it a try. Note that a related discussion appeared already years ago in~\cite[Section 2]{MSV2}. The reader not interested in torturing her or his mind can safely jump right away to Section \ref{sec:backtobusiness}.

A fairly easy direction is to realize Schur--Weyl duality inside skew Howe duality. To do so, let us fix $N=l$, and write, following Equation~\ref{eq:wedge+}:

\[
{\bigwedge}^l_q(\C_q^m\otimes \C_q^l)\simeq \bigoplus_{a_1+\cdots + a_l=l} {\bigwedge}_q^{a_1}(\C_q^m)\otimes \cdots \otimes {\bigwedge}_q^{a_l}(\C_q^m)
\]
The quantum Schur algebra $\calS_q(l,l)$ acts on this space by $U_q(\gl_m)$ intertwiners, and the space on the rhs contains $\C_q^m\otimes \cdots \otimes \C_q^m$ as a subspace (choosing $a_1=\dots = a_l=1$). $U_q(\gl_m)$ intertwiners for this space are realized by $\calH_l$, and thus one can embed $\calH_l\subset \calS_q(l,l)$. That way, one can somehow see Schur--Weyl duality as living inside skew Howe duality.

To go the other way around will require a little bit more of torturing things. Recall $\catRep^{\uparrow}$ is the category of $U_q(\gl_m)$ intertwiners monoidally generated by $V$, and that one can realize $\calH_N:=\End_{U_q(\gl_m)}(V^{\otimes N})$. Recall also the maps $\iota$ and $\pi$ allowing to realize ${\bigwedge}_q^a(\C_q^m)$ as a subrepresentation of $(\C_q^m)^{\otimes a}$. One can build over them to realize:
\[
{\bigwedge}_q^{a_1}(\C_q^m)\otimes \cdots \otimes {\bigwedge}_q^{a_l}(\C_q^m)\subset (\C_q^m)^{\otimes N}
\]

Now, the situation is as follows: $\calH_N$ acts on $(\C_q^m)^{\otimes N}$, which is the only space we have at hand, but which contains tensor products of exterior powers as subspaces. So what we need is to be able to cut that space in pieces, which is handily allowed by the process of taking a Karoubi envelope.

\begin{definition}
  The Karoubi envelope $\Kar(\mathcal{C})$ of a category $\mathcal{C}$ is the category whose objects are pairs $(M,p)$ with $M$ an object of $\mathcal{C}$ and $p\in \End_{\mathcal{C}}(M)$ an idempotent. For morphisms,
\[
\Hom_{\Kar(\mathcal{C})}((M_1,p_1),(M_2,p_2))=p_2\circ \Hom_{\mathcal{C}}(M_1,M_2)\circ p_1
\]
\end{definition}

In order to pass from $\calH_N$ to $\calS_q(l,N)$, we'll first transform $\calH_N$ into a category $\dot{\calH}_N$ with a single object (that we will think of as being $V^{\otimes N}$) and with endomorphism space over that object being $\calH_N$. So far, this is purely formal.

The Karoubi envelope $\Kar(\dot{\calH}_N)$ contains more objects, and in particular it contains objects corresponding to the projectors $V^{\otimes N}\twoheadrightarrow {\bigwedge}_q^{a_1}\otimes\cdots \otimes {\bigwedge}_q^{a_l}$. The point here is that these projectors live already in $\calH_N$. Let's denote such a projector by $\1_{[a_1,\dots, a_l]}$.

Let us now define Lusztig's idempotented version of the Schur algebra, $\dot{\calS}_q(l,N)$, as the full subcategory of $\Kar(\mathcal{C})$ with objects corresponding to the previous projectors. In other words:
\begin{itemize}
\item objects in $\dot{\calS}_q(l,N)$ are $\1_{[a_1,\dots,a_l]}$ with $\sum_i a_i=N$ and $a_i\geq 0$;
\item morphisms are given by:
  \[
  \Hom_{\dot{\calS}_q(l,N)}(\1_{[a_1,\dots,a_l]},\1_{[a'_1,\dots,a'_l]})=\1_{[a'_1,\dots,a'_l]}\calH_N\1_{[a_1,\dots,a_l]}
  \]
\end{itemize}
One can check (see \cite[Theorem 2.4]{DotyGiaquinto}):
\[
\mathcal{S}_q(l,N)=\bigoplus_{\substack{a_1+\dots +a_l=N\\a'_1+\dots + a'_l=N}} \Hom_{\dot{\calS}_q(l,N)}(\1_{[a_1,\dots,a_l]},\1_{[a'_1,\dots,a'_l]})
\]

In the end, one can get (Lusztig's idempotented version of) the Schur algebra as a full subcategory of the Karoubi envelope of the Hecke algebra, letting thus skew Howe duality live inside Schur--Weyl.

Before going back to the question of computing knot invariants and relating them to quantum algebras, let us just comment on the fact that the objects that appear in the category version of the Schur algebras are weight spaces for the dual $U_q(\gl_m)$ action. Using the definition of the quantum exterior power and the explicit action induced from the vector representation, one can very explicitly compute that:
\begin{gather} \label{eq:actionWSpaces}
  {K_i}_{|{\bigwedge}_q^{a_1}(V)\otimes \cdots \otimes {\bigwedge}_q^{a_l}(V)}=q^{a_i-a_{i+1}}id \\ E_i\left({\bigwedge}_q^{a_1}(V)\otimes \cdots \otimes {\bigwedge}_q^{a_l}(V)\right)\subset {\bigwedge}_q^{a_1}(V)\otimes \cdots\otimes {\bigwedge}_q^{a_{i}+1}(V)\otimes {\bigwedge}_q^{a_{i+1}-1}(V)\otimes \cdots \otimes {\bigwedge}_q^{a_l}(V) \nonumber
  \\ F_i\left({\bigwedge}_q^{a_1}(V)\otimes \cdots \otimes {\bigwedge}_q^{a_l}(V)\right)\subset {\bigwedge}_q^{a_1}(V)\otimes \cdots\otimes {\bigwedge}_q^{a_{i}-1}(V)\otimes {\bigwedge}_q^{a_{i+1}+1}(V)\otimes \cdots \otimes {\bigwedge}_q^{a_l}(V) \nonumber
  \end{gather}

These relations transport into the Schur algebras:
\begin{gather} \label{eq:SchurAlgRel} 
  {K_i}\1_{[a_1,\dots,a_l]}=q^{a_i-a_{i+1}}\1_{[a_1,\dots,a_l]} \\
  E_i\1_{[\dots,a_{i-1},a_i,a_{i+1},a_{i+2},\dots]} = \1_{[\dots,a_{i-1},a_i+1,a_{i+1}-1,\dots]}E_i   \\
F_i\1_{[\dots,a_{i-1},a_i,a_{i+1},a_{i+2},\dots]} = \1_{[\dots,a_{i-1},a_i-1,a_{i+1}+1,\dots]}F_i  
  \end{gather}

\subsection{Back to business} \label{sec:backtobusiness}

Recall that skew Howe duality was supposed to help us go from braids to more general tangles, including links. Let us restart from the general situation of the category $\catRep^{+}(U_q(\gl_{m|n}))$, generated by tensor products of exterior powers of the vector representation. Skew-Howe duality provides us a functor, the action of which can be read on objects from Equation~\eqref{eq:actionWSpaces}:
\begin{gather*}
  \dot{\calS}_q(l,N)\rightarrow \catRep^+(U_q(\gl_{m|n})) \\
  \1_{[a_1,\dots,a_l]}\rightarrow {\bigwedge}_q^{a_1}(V)\otimes \cdots \otimes {\bigwedge}_q^{a_l}(V)
\end{gather*}

Recall that $\calS_q(l,N)$ was obtained as a quotient of $U_q(\gl_l)$, and thus inherits a presentation by generators and relations.

Let us now assume that $n=0$, and consider the restriction functor:
\[
\catRep^{+}(U_q(\gl_{m}))\rightarrow \catRep^{+}(U_q(\slnn{m}))
\]

\begin{lemma}
  The following special cases hold in $\catRep^{+}(U_q(\slnn{m}))$.
  \[
    {\bigwedge}_q^m(V)\simeq \C_q \; \text{and}\; {\bigwedge}_q^{m-1}(V)\simeq V^{\ast}
    \]
\end{lemma}

\begin{proof}[Sketched proof]
  The first statement is an explicit computation. For weight reasons, only the $K_i$'s may act by non-zero operators, and Equation~\eqref{eq:actionWSpaces} shows that the value is always $1$. This identifies the top exterior power with the trivial representation, and then the natural projection map:
  \[
  V\otimes {\bigwedge}_q^{m-1}(V) \rightarrow {\bigwedge}_q^m{V}\simeq \C_q
  \]
 realizes the duality.
\end{proof}

More generally, one can check that ${\bigwedge}_q^k(V)\simeq ({\bigwedge}_q^{m-k}(V))^\ast$.

Here comes the upshot: $\catRep^+(U_q(\slnn{m}))$ contains all the objects needed to construct the Reshetikhin-Turaev invariant for tangles, and encodes in an algebraically rigid fashion the $\Hom$-spaces between such objects.

Very concretely, we are looking for a map that to elementary tangles assigns some morphism in the Schur category. This is exactly what~\cite{CKM} does.

\begin{theorem}
  The following rules are part of the definition of the functor from the category of tangles to the Schur category.
\[
\begin{tikzpicture}[anchorbase,scale=.5]
  \draw [<-] (1,1) -- (1,.5) to [out=-90,in=-90] (0,.5) -- (0,1);
\end{tikzpicture}
\quad \leftrightarrow \quad
\begin{tikzpicture}[anchorbase,scale=.5]
  \draw [dotted] (0,0) -- (0,.4);
  \draw [dotted] (1,0) -- (1,.6);
  \draw [semithick] (0,.4) -- (0,1);
  \draw [semithick] (1,.6) -- (1,1);
  \draw [semithick] (0,.4) -- (1,.6);
\end{tikzpicture}
\quad \leftrightarrow \quad
\1_{[m-1,1]}F\1_{[m,0]}
\]
\[
\begin{tikzpicture}[anchorbase,scale=.5]
  \draw [->] (1,1) -- (1,.5) to [out=-90,in=-90] (0,.5) -- (0,1);
\end{tikzpicture}
\quad \leftrightarrow \quad
\begin{tikzpicture}[anchorbase,scale=.5]
  \draw [dotted] (0,0) -- (0,.6);
  \draw [dotted] (1,0) -- (1,.4);
  \draw [semithick] (0,.6) -- (0,1);
  \draw [semithick] (1,.4) -- (1,1);
  \draw [semithick] (0,.6) -- (1,.4);
\end{tikzpicture}
\quad \leftrightarrow \quad
\1_{[1,m-1]}E\1_{[0,m]}
\]
\[
\begin{tikzpicture}[anchorbase,scale=.5]
  \draw [<-] (1,0) -- (1,.5) to [out=90,in=90] (0,.5) -- (0,0);
\end{tikzpicture}
\quad \leftrightarrow \quad
\begin{tikzpicture}[anchorbase,scale=.5]
  \draw [semithick] (0,0) -- (0,.4);
  \draw [semithick] (1,0) -- (1,.6);
  \draw [dotted] (0,.4) -- (0,1);
  \draw [dotted] (1,.6) -- (1,1);
  \draw [semithick] (0,.4) -- (1,.6);
\end{tikzpicture}
\quad \leftrightarrow \quad
\1_{[0,m]}F\1_{[1,m-1]}
\]
\[
\begin{tikzpicture}[anchorbase,scale=.5]
  \draw [->] (1,0) -- (1,.5) to [out=90,in=90] (0,.5) -- (0,0);
\end{tikzpicture}
\quad \leftrightarrow \quad
\begin{tikzpicture}[anchorbase,scale=.5]
  \draw [semithick] (0,0) -- (0,.6);
  \draw [semithick] (1,0) -- (1,.4);
  \draw [dotted] (0,.6) -- (0,1);
  \draw [dotted] (1,.4) -- (1,1);
  \draw [semithick] (0,.6) -- (1,.4);
\end{tikzpicture}
\quad \leftrightarrow \quad
\1_{[m,0]}E\1_{[m-1,1]}
\]
\end{theorem}

In the above theorem, the middle pictures serve as an intermediate step between the topological data and its algebraic counterpart. They actually generalize to so-called ladder instances of the more flexible idea of webs, at the heart of a fruitful history at the interface between quantum topology and representation theory. But let us not say more about that, and refer to~\cite{CKM} for more details and references.

We are still lacking a piece of data to conclude: we need to say what gets assigned to crossings. The formulas for the kind of crossings that appears in the story we have been studying so far (that is, the case where the strands of the knots carry no colors and thus correspond to vector representations) are easy to write down, but it is worth noticing that it is part of a larger story that relates to the so-called ``quantum Weyl group action'' defined by Lusztig~\cite[Chapters 5 and 37]{Lus4}.

\begin{theorem}
  The following rules give the image of a crossing in the Schur category.
  \[
  \begin{tikzpicture}[smallnodes,anchorbase,xscale=0.7,yscale=0.5]
   \draw[uno] (1,0)  -- ++(0,0.3) \midarrow .. controls ++(0,0.7) and ++(0,-0.7) .. ++(-1,1.4) -- ++(0,0.3)  \midarrow ;
   \draw[uno,cross line] (0,0)  -- ++(0,0.3) \midarrow .. controls ++(0,0.7) and ++(0,-0.7) .. ++(1,1.4) -- ++(0,0.3)  \midarrow ;
     \end{tikzpicture}
\quad  \leftrightarrow \quad
q^{-1}\1_{[1,1]} -FE\1_{[1,1]}\in \End_{\dot{\calS}_q}(\1_{[1,1]})
\]
and
  \[
  \begin{tikzpicture}[smallnodes,anchorbase,xscale=0.7,yscale=0.5]
   \draw[uno] (0,0)  -- ++(0,0.3) \midarrow .. controls ++(0,0.7) and ++(0,-0.7) .. ++(1,1.4) -- ++(0,0.3)  \midarrow ;
   \draw[uno, cross line] (1,0)  -- ++(0,0.3) \midarrow .. controls ++(0,0.7) and ++(0,-0.7) .. ++(-1,1.4) -- ++(0,0.3)  \midarrow ;
     \end{tikzpicture}
\quad  \leftrightarrow \quad
q\1_{[1,1]} -FE\1_{[1,1]}\in \End_{\dot{\calS}_q}(\1_{[1,1]})
\]

\end{theorem}

All statements above are to be considered as local. The general recipe to turn a tangle or link diagram into an element of the Schur algebra would be to isotope it in a more rigid-looking way, so that it can be sliced into element pieces which consist of one of the above element tangles with extra vertical strands. That way, one can turn the diagram into a composition of elements in the Schur algebras extracted from the previous theorem, and use the algebraic power of those to run the computations.

\begin{example} \label{ex:unknot}
  A very easy example, but one I actually quite like, is to consider the unknot:
  
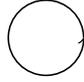
\[
  \begin{tikzpicture}[anchorbase]
    \draw [semithick,->] (0,0) arc (0:360:.5);
    \end{tikzpicture}
  \]
  Presenting it as a composition of a cap with a cup, it translates in $\End_{\dot{\calS}_q(1,m)}(\1_{[0,m]})$ into $FE\1_{[0,m]}$. Now, recall that:
\[
FE\1_{[0,m]}=EF\1_{[0,m]}-\frac{K-K^{-1}}{q-q^{-1}}\1_{[0,m]}
\]
and that $EF\1_{[0,m]}$, starting from an object corresponding to ${\bigwedge}_q^0(\C_q^m)\otimes {\bigwedge}_q^m(\C_q^m)$, factors through ${\bigwedge}_q^{-1}(\C_q^m)\otimes {\bigwedge}_q^{m+1}(\C_q^m)$, which is zero. Furthermore, from Equation~\eqref{eq:SchurAlgRel}, one reads that $K\1_{[0,m]}=q^{-m}$. Thus one gets:
 \[
  \begin{tikzpicture}[anchorbase]
    \draw [semithick,->] (0,0) arc (0:360:.5);
    \end{tikzpicture}
\rightarrow -\frac{q^{-m}-q^m}{q-q^{-1}}=[m]
  \]
 as desired!
\end{example}

\begin{example}
  We consider the Hopf link, and aim at computing its $\slnn{m}$ Reshetikhin-Turaev invariant. It can be rigidified as follows:
  
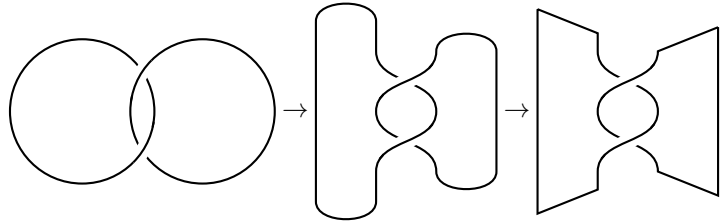
\[
\begin{tikzpicture}[anchorbase,scale=.8]
  \begin{knot}[
            clip width=7,
            consider self intersections,
            end tolerance=1pt,
            flip crossing=1,
    ]
    \strand [thick] (0,0) circle (1.2);
    \strand [thick] (2,0) circle (1.2);
    \end{knot}
\end{tikzpicture}
\rightarrow
\begin{tikzpicture}[anchorbase,scale=.8]
  \begin{knot}[
            clip width=7,
            consider self intersections,
            end tolerance=1pt,
            flip crossing=1,
    ]
    \strand [thick]  (0,2.5) to [out=90,in=90] (1,2.5) to (1,2) to [out=-90,in=90] (2,1) to [out=-90,in=90] (1,0) to (1,-.5) to [out=-90,in=-90] (0,-.5) to (0,2.5);
    \strand [thick] (3,2) to [out=90,in=90] (2,2) to [out=-90,in=90] (1,1) to [out=-90,in=90] (2,0) to [out=-90,in=-90] (3,0) to (3,2);
    \end{knot}
\end{tikzpicture}
\rightarrow
\begin{tikzpicture}[anchorbase,scale=.8]
  \begin{knot}[
            clip width=7,
            consider self intersections,
            end tolerance=1pt,
            flip crossing=1,
    ]
    \strand [thick]  (0,2.7) -- (1,2.3) to (1,2) to [out=-90,in=90] (2,1) to [out=-90,in=90] (1,0) to (1,-.3) -- (0,-.7) to (0,2.7);
    \strand [thick] (3,2.4) -- (2,2) to [out=-90,in=90] (1,1) to [out=-90,in=90] (2,0) -- (3,-.4) to (3,2.4);
    \end{knot}
\end{tikzpicture}
\]
and it gets assigned the following element of $\End_{\dot{\calS}_q(1,m)}(\1_{[m,0,0,m]})$:
\begin{align*}
E_1F_3(q^{-1}-F_2E_2)^2E_3F_1\1_{[m,0,0,m]}=&q^{-2}E_1F_3E_3F_1\1_{[m,0,0,m]}\\ & -2q^{-1}E_1F_3F_2E_2E_3F_1\1_{[m,0,0,m]}\\ &+E_1F_3F_2E_2F_2E_2E_3F_1\1_{[m,0,0,m]}
\end{align*}
We can simplify this one term at a time:
\[
E_1F_3E_3F_1\1_{[m,0,0,m]}=F_3E_3E_1F_1\1_{[m,0,0,m]}
\]
which using the same computations as in Example \ref{ex:unknot} simply reduces to $[m]^2$.
Now, the last term actually reduces to the middle one:
\begin{align*}
E_1F_3F_2E_2F_2E_2E_3F_1\1_{[m,0,0,m]}&=E_1F_3F_2E_2F_2\1_{[m-1,2,0,m-1]}E_2E_3F_1\1_{[m,0,0,m]} \\
&= E_1F_3F_2F_2E_2\1_{[m-1,2,0,m-1]}E_2E_3F_1\1_{[m,0,0,m]} \\
&\; +E_1F_3F_2\frac{K_2-K_2^{-1}}{q-q^{-1}}\1_{[m-1,2,0,m-1]}E_2E_3F_1\1_{[m,0,0,m]} \\
\end{align*}
The first part of the result dies because $E_2\1_{[m-1,2,0,m-1]}=\1_{[m-2,3,-1,m-1]}E_2=0$, and the remaining part reduces to:
\[[2]E_1F_3F_2E_2E_3F_1\1_{[m,0,0,m]}\]
We are left with the following computation:
\begin{align*}
  ([2]-2q^{-1})E_1F_3F_2E_2E_3F_1\1_{[m,0,0,m]} &=  (q-q^{-1})F_3F_2E_1F_1\1_{[m,1,0,m-1]}E_2E_3 \\
  &= F_3F_2F_1\1_{[m+1,0,0,m-1]}E_1+[m-1]F_3F_2E_2\1_{[m,0,1,m-1]}E_3
\end{align*}
The first term above dies if we let it act on:
\[
  {\bigwedge_q}^m(\C_q^m)\otimes{\bigwedge_q}^0(\C_q^m)\otimes{\bigwedge_q}^0(\C_q^m)\otimes{\bigwedge_q}^m(\C_q^m)
  \]
  as it factors through
\[
  {\bigwedge_q}^{m+1}(\C_q^m)\otimes{\bigwedge_q}^0(\C_q^m)\otimes{\bigwedge_q}^0(\C_q^m)\otimes{\bigwedge_q}^m(\C_q^{m-1})=\{0\}
  \]
  Note that this is very specific to the $\slnn{m}$ case!
  Then:
  \begin{align*}
    ([2]-2q^{-1})E_1F_3F_2E_2E_3F_1\1_{[m,0,0,m]} &=  (q-q^{-1})[m-1]F_3E_3\1_{[m,0,0,m]}\\
    &=(q-q^{-1})[m][m-1]\1_{[m,0,0,m]}
  \end{align*}
  Finally, one computes that the $\slnn{m}$ invariant of the Hopf link equals:
  \[q^{-2}[m]^2+ (q-q^{-1})[m][m-1]=[m](q[m-1]+q^{-m-1})\]
\end{example}

Surely, all of these computations could have been performed directly in the Brauer category. However, in addition to the arguments in favor of the Schur algebras given in the next paragraph, the idea of using PBW algorithm or similar in the context of Schur algebras, or similar algebra-originating algorithmic processes, might be an interesting computational tool.

\subsection{Restrictions and how to overcome them}

First of all, the process that turns a tangle diagram into something that corresponds to an element in the Schur algebra is a bit unsatisfactory, especially if one is willing to go to the categorical level, as it somehow consists in reversing some of the orientations, which is terrible for functoriality purposes. Perhaps more importantly, the process dramatically fails if $V^\ast\notin \Kar(\catRep^+)$, which happens for example in the super case. Similarly, in the HOMFLY-PT case, one even lacks the representation category to start from.

These issues motivated the introduction with Antonio Sartori of the notion of doubled Schur algebra \cite{QS2}, which is based on skew Howe duality with duals, and is closely related to the Brauer category at generic parameter $\beta$, thus giving a quantum description of the HOMFLY-PT polynomial. This approach was later categorified by Naisse and Vaz~\cite{NaisseVaz}.

Another interesting question, very much related to the kind of categories that Reshetikhin and Turaev have used to pass from the 3-sphere to more general 3-manifolds, is to replace exterior powers by symmetric powers. Answers to these questions can be found in~\cite{RTub,TVW}. In many of these cases though, and especially if one goes to the super case, the generator and relation description of what replaces the Schur category is still hard to handle.

\bibliographystyle{plain}

\begin{thebibliography}{10}

\bibitem{Alexander}
J.~W. Alexander.
\newblock Topological invariants of knots and links.
\newblock {\em Trans. Amer. Math. Soc.}, 30(2):275--306, 1928.

\bibitem{BerensteinZwicknagl}
A.~Berenstein and S.~Zwicknagl.
\newblock Braided symmetric and exterior algebras.
\newblock {\em Trans. Amer. Math. Soc.}, 360(7):3429--3472, 2008.
\newblock \href{https://arxiv.org/abs/math/0504155}{arXiv:math/0504155}.

\bibitem{BCGP}
C.~Blanchet, F.~Costantino, N.~Geer, and B.~Patureau-Mirand.
\newblock Non-semi-simple tqfts, reidemeister torsion and kashaev's invariants.
\newblock {\em Advances in Mathematics}, 301:1--78, 2016.
\newblock \href{https://arxiv.org/abs/1404.7289}{arXiv:1404.7289}.

\bibitem{BurdeZieschang}
G.~Burde and H.~Zieschang.
\newblock {\em Knots}, volume~5 of {\em de Gruyter Studies in Mathematics}.
\newblock Walter de Gruyter \& Co., Berlin, 1985.

\bibitem{CKM}
S.~Cautis, J.~Kamnitzer, and S.~Morrison.
\newblock Webs and quantum skew {H}owe duality.
\newblock {\em Math. Ann.}, 360(1-2):351--390, 2014.
\newblock \href{https://arxiv.org/abs/1210.6437}{arXiv:1210.6437}.

\bibitem{ChariPressley}
V.~Chari and A.~Pressley.
\newblock {\em A guide to quantum groups}.
\newblock Cambridge University Press, Cambridge, 1994.

\bibitem{ChengWang}
S-J. Cheng and W.~Wang.
\newblock Howe duality for {L}ie superalgebras.
\newblock {\em Compositio Math.}, 128(1):55--94, 2001.
\newblock \href{https://arxiv.org/abs/math/0008093}{arXiv:math/0008093}.

\bibitem{CGP}
F.~Costantino, N.~Geer, and B.~Patureau-Mirand.
\newblock Quantum invariants of 3-manifolds via link surgery presentations and
  non-semi-simple categories.
\newblock {\em Journal of Topology}, 7(4):1005--1053, 2014.
\newblock \href{https://arxiv.org/abs/1202.3553}{arXiv:1202.3553}.

\bibitem{Cromwell}
P.~R Cromwell.
\newblock {\em Knots and links}.
\newblock Cambridge University Press, 2004.

\bibitem{DDS}
R.~Dipper, S.~Doty, and F.~Stoll.
\newblock The quantized walled {B}rauer algebra and mixed tensor space.
\newblock {\em Algebr. Represent. Theory}, 17(2):675--701, 2014.
\newblock \href{https://arxiv.org/abs/0806.0264}{arXiv:0806.0264}.

\bibitem{DotyGiaquinto}
S.~Doty and A.~Giaquinto.
\newblock Presenting {S}chur algebras.
\newblock {\em Int. Math. Res. Not.}, 2002(36):1907--1944, 2002.
\newblock \href{https://arxiv.org/abs/math/0108046}{arXiv:math/0108046}.

\bibitem{Etingof}
P.~I Etingof, O.~Golberg, S.~Hensel, T.~Liu, A.~Schwendner, D.~Vaintrob, and
  E.~Yudovina.
\newblock {\em Introduction to representation theory}, volume~59.
\newblock American Mathematical Soc., 2011.

\bibitem{HOMFLY}
P.~Freyd, D.~Yetter, J.~Hoste, W.~B.~R. Lickorish, K.~Millett, and A.~Ocneanu.
\newblock A new polynomial invariant of knots and links.
\newblock {\em Bull. Amer. Math. Soc. (N.S.)}, 12(2):239--246, 1985.

\bibitem{GP}
A.~M. Garsia and C.~Procesi.
\newblock On certain graded {$S_n$}-modules and the {$q$}-{K}ostka polynomials.
\newblock {\em Adv. Math.}, 94(1):82--138, 1992.

\bibitem{GP_supertrace}
N.~Geer and B.~Patureau-Mirand.
\newblock An invariant supertrace for the category of representations of {L}ie
  superalgebras of type {I}.
\newblock {\em Pacific J. Math.}, 238(2):331--348, 2008.
\newblock \href{https://arxiv.org/abs/0711.4231}{arXiv:0711.4231}.

\bibitem{Grant}
J.~{Grant}.
\newblock {A generators and relations description of a representation category
  of $U_q(\mathfrak{gl}(1|1))$}.
\newblock {\em Algebraic \& Geometric Topology}, 16(1):509--539, 2016.
\newblock \href{https://arxiv.org/abs/1406.5399}{arXiv:1406.5399}.

\bibitem{Howe}
R.~Howe.
\newblock Perspectives on invariant theory: {S}chur duality, multiplicity-free
  actions and beyond.
\newblock In {\em The {S}chur lectures (1992) ({T}el {A}viv)}, volume~8 of {\em
  Israel Math. Conf. Proc.}, pages 1--182. Bar-Ilan Univ., Ramat Gan, 1995.

\bibitem{Jimbo}
M.~Jimbo.
\newblock A q-analogue of u(gl(n+1)), hecke algebra, and the yang-baxter
  equation.
\newblock {\em Letters in Mathematical Physics}, 11(3):247--252, 1986.

\bibitem{Jones}
V.~F.~R. Jones.
\newblock A polynomial invariant for knots via von {N}eumann algebras.
\newblock {\em Bull. Amer. Math. Soc. (N.S.)}, 12(1):103--111, 1985.

\bibitem{Kassel}
C.~Kassel.
\newblock {\em Quantum groups}, volume 155 of {\em Graduate Texts in
  Mathematics}.
\newblock Springer-Verlag, New York, 1995.

\bibitem{KauffmanSaleur}
L.~H. Kauffman and H.~Saleur.
\newblock Free fermions and the {A}lexander-{C}onway polynomial.
\newblock {\em Comm. Math. Phys.}, 141(2):293--327, 1991.

\bibitem{Lus4}
G.~Lusztig.
\newblock {\em Introduction to quantum groups}, volume 110 of {\em Progress in
  Mathematics}.
\newblock Birkh\"auser Boston Inc., Boston, MA, 1993.

\bibitem{MSV2}
M.~Mackaay, M.~Sto{\v{s}}i{\'c}, and P.~Vaz.
\newblock A diagrammatic categorification of the q-{S}chur algebra.
\newblock {\em Quantum Topol.}, 4(1):1--75, 2013.
\newblock \href{http://arxiv.org/abs/1008.1348}{arXiv:1008.1348}.

\bibitem{Mitsuhashi}
H.~Mitsuhashi.
\newblock Schur-{W}eyl reciprocity between the quantum superalgebra and the
  {I}wahori-{H}ecke algebra.
\newblock {\em Algebr. Represent. Theory}, 9(3):309--322, 2006.
\newblock \href{https://arxiv.org/abs/math/0506156}{arXiv:math/0506156}.

\bibitem{MurakamiAlexander}
J.~Murakami.
\newblock {The multi-variable Alexander polynomial and a one-parameter family
  of representations of ${\mathcal U}\sb q({\mathfrak sl}(2,\mathbb C))$ at
  $q\sp 2 = -1$.}
\newblock In {\em {Quantum groups}}, pages 350--353. Springer Berlin
  Heidelberg, 1992.

\bibitem{NaisseVaz}
G.~Naisse and P.~Vaz.
\newblock 2-{V}erma modules and the {K}hovanov-{R}ozansky link homologies.
\newblock 2017.
\newblock \href{https://arxiv.org/abs/1704.08485}{arXiv:1704.08485}.

\bibitem{PT}
J.~H. Przytycki and P.~Traczyk.
\newblock Conway algebras and skein equivalence of links.
\newblock {\em Proc. Amer. Math. Soc.}, 100(4):744--748, 1987.

\bibitem{QS2}
H.~{Queffelec} and A.~{Sartori}.
\newblock {Mixed quantum skew Howe duality and link invariants of type A}.
\newblock {\em Journal of Pure and Applied Algebra}, 223(7):2733--2779, 2019.
\newblock \href{https://arxiv.org/abs/1504.01225}{arXiv:1504.01225}.

\bibitem{RT}
N.~Y. Reshetikhin and V.~G. Turaev.
\newblock Ribbon graphs and their invariants derived from quantum groups.
\newblock {\em Comm. Math. Phys.}, 127(1):1--26, 1990.

\bibitem{Rolfsen}
D.~Rolfsen.
\newblock {\em Knots and links}, volume 346.
\newblock American Mathematical Soc., 2003.

\bibitem{RTub}
D.~E.~V. {Rose} and D.~{Tubbenhauer}.
\newblock {Symmetric webs, Jones-Wenzl recursions and $q$-Howe duality}.
\newblock {\em International Mathematics Research Notices},
  2016(17):5249--5290, 2015.
\newblock \href{http://arxiv.org/abs/1501.00915}{arXiv:1501.00915}.

\bibitem{SarAlexander}
A.~Sartori.
\newblock The {A}lexander polynomial as quantum invariant of links.
\newblock {\em Ark. Mat.}, 53(1):177--202, 2015.
\newblock \href{https://arxiv.org/abs/1308.2047}{arXiv:1308.2047}.

\bibitem{Sartori_cat}
A.~{Sartori}.
\newblock {Categorification of tensor powers of the vector representation of
  $U_q(\mathfrak{gl}(1|1))$}.
\newblock {\em Selecta Mathematica}, 22(2):669--734, 2016.
\newblock \href{https://arxiv.org/abs/1305.6162}{arXiv:1305.6162}.

\bibitem{Scheunert}
M.~Scheunert.
\newblock Serre-type relations for special linear lie superalgebras.
\newblock {\em Letters in Mathematical Physics}, 24(3):173--181, 1992.

\bibitem{TVW}
D.~Tubbenhauer, P.~Vaz, and P.~Wedrich.
\newblock Super {$q$}-{H}owe duality and web categories.
\newblock {\em Algebr. Geom. Topol.}, 17(6):3703--3749, 2017.
\newblock \href{https://arxiv.org/abs/1504.05069}{arXiv:1504.05069}.

\bibitem{ViroAlexander}
O.Y. Viro.
\newblock {Quantum relatives of the Alexander polynomial}.
\newblock {\em St. Petersbg. Math. J.}, 18(3):391--457, 2007.
\newblock \href{http://arxiv.org/abs/math/0204290}{arXiv:math/0204290}.

\bibitem{Witten3}
E.~Witten.
\newblock Quantum field theory and the {J}ones polynomial.
\newblock {\em Comm. Math. Phys.}, 121(3):351--399, 1989.

\bibitem{Witten}
E.~Witten.
\newblock Fivebranes and knots.
\newblock {\em Quantum Topol.}, 3(1):1--137, 2012.
\newblock \href{https://arxiv.org/abs/1101.3216}{arXiv:1101.3216}.

\bibitem{WuZ}
Y.~Wu and R.~B. Zhang.
\newblock Unitary highest weight representations of quantum general linear
  superalgebra.
\newblock {\em J. Algebra}, 321(11):3568--3593, 2009.

\bibitem{Zhang_super_qgroup}
R.~B. Zhang.
\newblock Structure and representations of the quantum general linear
  supergroup.
\newblock {\em Communications in mathematical physics}, 195(3):525--547, 1998.
\newblock \href{http://arxiv.org/abs/q-alg/9611024}{arXiv:q-alg/9611024}.

\bibitem{Zhang}
R.~B. Zhang.
\newblock Quantum enveloping superalgebras and link invariants.
\newblock {\em J. Math. Phys.}, 43(4):2029--2048, 2002.

\end{thebibliography}

\end{document}